\documentclass[notitlepage,onecolumn,amsmath,nobibnotes,nofootinbib,superscriptaddress,preprintnumbers]{revtex4-1}
\usepackage{hyperref}
\hypersetup{breaklinks,colorlinks,citecolor=blue}

\usepackage{dsfont}
\usepackage{amssymb}
\usepackage{amsmath}
\usepackage{mathrsfs}
\usepackage{amsthm}
\usepackage{color}
\usepackage{xcolor}
\usepackage{bold-extra}

\numberwithin{equation}{section}
\newtheorem{lemma}{Lemma}[section]
\newtheorem{definition}{Definition}[section]
\newtheorem{proposition}{Proposition}[section]
\newtheorem{corollary}{Corollary}[section]
\newtheorem{theorem}{Theorem}[section]
\newtheorem{remark}{Remark}[section]
\newtheorem*{openq*}{\color{red}{Open Question}}

\newtheorem{example}{Example}[section]

\newcommand{\R}{\mathbb{R}}

\begin{document}

\title{\Large \bfseries{\scshape{Gaussian approximations of nonlinear statistics on the sphere}}}

\author{Solesne Bourguin}
\email{solesne.bourguin@gmail.com}
\affiliation{Department of Mathematical Sciences, Carnegie Mellon University, Pittsburgh, USA}
\author{Claudio Durastanti}
\email{durastan@mat.uniroma2.it}
\thanks{supported by the ERC grant \textit{Pascal} 277742}
\affiliation{Department of Mathematics, University of Rome Tor Vergata, Rome, Italy}
\author{Domenico Marinucci}
\email{marinucc@mat.uniroma2.it}
\thanks{partially supported by the ERC grant \textit{Pascal} 277742}
\affiliation{Department of Mathematics, University of Rome Tor Vergata, Rome, Italy}
\author{Giovanni Peccati}
\email{giovanni.peccati@gmail.com}
\thanks{partially supported by the grant F1R-MTH-PUL-12PAMP (PAMPAS)}
\affiliation{Unit\'e de Recherche en Math\'ematiques, Luxembourg University, Luxembourg, Luxembourg}

 \begin{abstract}
~\\
\noindent \textit{Abstract:} We show how it is possible to assess the rate of convergence in the Gaussian approximation of triangular arrays of $U$--statistics, built from wavelets coefficients
evaluated on a homogeneous spherical Poisson field of arbitrary dimension.
For this purpose, we exploit the Stein--Malliavin approach introduced in the
seminal paper by Peccati, Sol\'{e}, Taqqu and Utzet (2011); we focus in
particular on statistical applications covering evaluation of variance in
non--parametric density estimation and Sobolev tests for uniformity.
~
 \end{abstract}


\maketitle
\section{Introduction and Overview}

\subsection{Motivations}

\noindent The purpose of this paper is to establish quantitative central limit
theorems for some $U$-statistics on wavelets coefficients evaluated either on
spherical Poisson fields or on a vector of independent and identically distributed (i.i.d.) observations with values on a sphere. These statistics are motivated by standard
problems in statistical inference, such as evaluation of the variance in
density estimations and Sobolev tests of uniformity of the underlying
Poisson measure. Such problems are certainly very classical in statistical
inference; however, we shall investigate their solution under circumstances
which are somewhat non-standard, for a number of reasons. In particular, we
will focus mainly on ``high-frequency" procedures, where the scale to be
investigated and the number of tests to be implemented are themselves a
function of the number of observations available, according to rules to be
discussed below; for these statistics, we shall establish quantitative
central limit theorems by means of the so-called {\it Malliavin-Stein technique}.
Such a technique will allow, for instance, to determine how many joint
procedures can be run while maintaining a given level of accuracy in the
Gaussian approximation for the sample distribution of the resulting
statistics; as shown in Sections \ref{ss:1} and \ref{ss:2} below, a refined version of an argument contained in the classic paper by Dynkin and Mandelbaum \cite{DyMa} will allow us to extend our quantitative result (in a fully multidimensional setting) to the framework of $U$-statistics based on i.i.d. spherical observations. 

\medskip

\noindent As already mentioned, we shall assume that the domain of interest is the
unit sphere $\mathbb{S}^{q}\subset \mathbb{R}^{q+1}$. The arguments we
exploit can be extended to other compact manifolds, but we shall not pursue
these generalizations here for brevity and simplicity; however, on the
contrary of most of the existing literature, our procedures can also be
easily adapted to cover "local" tests, i.e. the possibility that these
spheres are only partially observable, as it is often the case for instance
in astrophysical experiments, cfr. for instance \cite{starckbook}, see also
the recent monograph \cite{BandB} for several other applications of
spherical data analysis.

\medskip

\noindent Malliavin-Stein techniques for Poisson processes have recently drawn a lot
of attention in the probabilistic literature, see for instance \cite{BouPec,
LPST, lp2013, np-ptrf, PSTU, lesmathias}, as well as the textbooks \cite{np-book} and 
\cite{ChenGoldShao} for background results on Gaussian approximations by means
Stein's method. As motivated above, our aim here is to apply and extend the
now well-known results of \cite{PSTU, PecZheng} in order to deduce bounds
that are well-adapted to the applications we mentioned; our principal
motivation originates from the implementation of wavelet systems on the
sphere in the framework of statistical analysis for Cosmic Rays data, as for
instance in \cite{iuppa, kerkyphampic, scodeller2, starck}. As noted in \cite%
{dmp2014}, under these circumstances, when more and more data become
available, higher and higher frequencies (i.e., smaller and smaller scales)
can be probed. We shall hence be concerned with sequences of Poisson fields,
whose intensity grows monotonically; it is then possible to exploit local
Normal approximations, where the rate of convergence to the asymptotic
Gaussian distribution is related to the scale parameter of the corresponding
wavelet transform in a natural and intuitive way. Similar arguments were
earlier exploited for linear statistics in \cite{dmp2014}; the proofs in the
nonlinear case we consider here are considerably more complicated from the
technical point of view, but remarkably the main qualitative conclusions go
through unaltered.

\subsection{U-Statistics on the Poisson Space}

\noindent We will now recall a few basic definitions on Poisson random measures and
Stein-Malliavin bounds; we refer for instance to \cite{PeTa, privaultbook,
SW} for more discussions and details. Assuming that we are working on a
suitable probability space $\left( \Omega ,\mathcal{F},\mathbb{P}\right) $,
the following definition is standard:

\begin{definition}
\label{d:poisrm} Let $\left( \Theta ,\mathcal{A},\lambda \right) $ be a $%
\sigma $-finite measure space, and assume that $\lambda $ has no atoms (that
is, $\lambda \left( \left\{ x\right\} \right) =0$, for every $x\in \Theta $%
). A Poisson random measure on $\Theta $ with intensity measure (or
control measure) $\lambda $ is a collection of random variables $\left\{
N\left( A\right) \colon A\in \mathcal{A}\right\} ,$ taking values in $\mathbb{Z}%
_{+}\cup \left\{ +\infty \right\} $, such that the following two properties
hold:

\begin{enumerate}
\item $N\left( A\right) $ has Poisson distribution with mean $\lambda \left(
A\right) ,$ for every $A$ $\in \mathcal{A}$;

\item $N\left( A_{1}\right) ,\ldots ,N\left( A_{n}\right) $ are independent
whenever $A_{1},\ldots, A_{n}\in \mathcal{A}$ are pairwise disjoint.
\end{enumerate}
\end{definition}

\noindent In what follows, we shall consider a special case of Definition \ref{d:poisrm}; more
precisely, we take $\Theta =\mathbb{R}_{+}\times \mathbb{S}^{q}$, with $%
\mathcal{A}=\mathcal{B}(\Theta )$, the class of Borel subsets of $\Theta $.
The symbol $N$ indicates a Poisson random measure on $\Theta $, with
homogeneous intensity given by $\lambda =\rho \times \mu $. We shall take $%
\rho (ds)=R\cdot \ell (ds)$, where $\ell $ is the Lebesgue measure and $R>0$ is a fixed parameter, in such a way that $\rho ([0,t]) :=R_{t}=R\cdot t$. Also, we assume that $\mu $ is a probability on $\mathbb{S}^{q}$
of the form $\mu (dx)=f(x)dx$, where $f$ is a density on the sphere. Given such an object, we will denote by $N_t$ ($t>0$) the Poisson measure on $(\mathbb{S}^q, \mathcal{B}(\mathbb{S}^q))$ given by 
\begin{equation}\label{e:t}
N_t(B) := N([0,t]\times B), \quad B\in \mathcal{B}(\mathbb{S}^q);
\end{equation}
it is easy to verify that $N_t$ has control $\mu_t := R_t \mu$. 

\medskip

\noindent Let us also review some standard distances between laws of random variables
taking values in $\mathbb{R}^{q}$; the first two (Wasserstein and Kolmogorov distances) will
be only used in the univariate case. Given a function $g\in \mathcal{C}^{1}(%
\mathbb{R}^{q})$, we write $\Vert g\Vert _{Lip}=\sup\limits_{x\in \mathbb{R}%
^{q}}\Vert \nabla g(x)\Vert _{\mathbb{R}^{q}}$. If $g\in \mathcal{C}^{2}(%
\mathbb{R}^{q})$, we set 
\begin{equation*}
M_{2}(g)=\sup_{x\in \mathbb{R}^{q}}\Vert \mathrm{Hess}\ g(x)\Vert _{op},
\end{equation*}%
where $\Vert \cdot \Vert _{op}$ indicates the operator norm.

\begin{definition}
The Wasserstein distance $d_{W}$, between the laws
of two random vectors $X,Y$ with values in $\mathbb{R}^{q}$ ($q\geq 1$) and
such that $\mathbb{E}\left\Vert X\right\Vert _{\mathbb{R}^{q}},\mathbb{E}\left\Vert
Y\right\Vert _{\mathbb{R}^{q}}<\infty $, is given by:%
\begin{equation*}
d_{W}\left( X,Y\right) =\sup_{g\colon \Vert g\Vert _{Lip}\leq 1}\left\vert \mathbb{E}\left[
g\left( X\right) \right] -\mathbb{E}\left[ g\left( Y\right) \right] \right\vert.
\end{equation*}
\end{definition}
\begin{definition}
The Kolmogorov distance $d_{K}$, between the laws
of two random variables $X,Y$ with values in $\mathbb{R}$ and
such that $\mathbb{E} \vert X\vert, \mathbb{E}\vert
Y\vert <\infty $, is given by:
\begin{equation*}
d_{K}\left( X,Y\right) =\sup_{z \in \mathbb{R}}\left\vert \mathbb{P}\left[
X \leq z \right] -\mathbb{P}\left[
Y \leq z \right] \right\vert.
\end{equation*}
\end{definition}
\begin{definition}
The distance $d_{2}$ between the laws of two random
vectors $X,Y$ with values in $\mathbb{R}^{q}$ ($q\geq 1$), such that $%
\mathbb{E}\left\Vert X\right\Vert _{\mathbb{R}^{q}},\mathbb{E}\left\Vert Y\right\Vert _{%
\mathbb{R}^{q}}<\infty $, is given by:%
\begin{equation*}
d_{2}\left( X,Y\right) =\sup_{g\in \mathcal{H}}\left\vert \mathbb{E}\left[ g\left(
X\right) \right] -\mathbb{E}\left[ g\left( Y\right) \right] \right\vert ,
\end{equation*}%
where $\mathcal{H}$ denotes the collection of all functions $g\in \mathcal{C}%
^{2}\left( \mathbb{R}^{q}\right) $ such that $\Vert g\Vert _{Lip}\leq 1$ and 
$M_{2}(g)\leq 1$.
\end{definition}

\noindent The concept of a $U$-statistic was introduced in a seminal paper by Hoeffding \cite%
{hoeffding48}, and since then it has become a central notion for statistical
inference (see e.g., \cite{sen1992}). Let us recall a general definition,
following \cite{lesmathias}. 

\begin{definition}[$U$-statistics] Consider a Poisson random
measure $N$ with control $\nu$ on $(A,\mathcal{A})$.
Fix $k\geq 1$. A random variable $F$ is called a $U$-\textit{statistic of
order} $k$, based on the Poisson measure $N$ with control $\nu$, if
there exists a kernel $h\in L_{s}^{1}(\nu ^{k})$ such that 
\begin{equation}
F=\sum_{(x_{1},\ldots ,x_{k})\in N_{\neq }^{k}}h(x_{1},\ldots ,x_{k}),
\label{e:ustat}
\end{equation}%
where the symbol $N_{\neq }^{k}$ indicates the class of all $k$-dimensional
vectors $(x_{1},\ldots ,x_{k})$ such that $x_{i}\in N$ and $x_{i}\neq
x_{j}$ for every $1\leq i\neq j\leq k$.
\end{definition}

\noindent As anticipated, in this paper we will focus, for all $t>0$, $U$--statistics on the $q$-dimensional sphere $\mathbb{S}^{q}$ based on the Poisson measure $N_t$ introduced in \eqref{e:t},
corresponding to the case $A= \mathbb{S}^{q}$, with $\mathcal{A}$ the associated Borel $\sigma$ field, and $\nu = \mu_t =R_t\mu$.
\\~\\
\noindent As discussed in the introduction, we shall consider two classical issues in
the statistical analysis of Poisson processes, namely estimation of variance
in density estimation and testing for uniformity of the governing measure.
In these two cases, a common form of statistic is 
\begin{equation}
U_{j}=\sum_{(z_{1},z_{2})\in N_{\neq }^{2}}h_{j}(z_{1},z_{2}),
\label{general0}
\end{equation}%
where $h_{j}(\cdot ,\cdot )$ is a kernel in the space $L^{2}\left( \mu
_{t}^{\otimes 2}\right) $.

\subsection{Spherical Needlets}

\noindent In this section we will provide a short overview about the construction of
needlet frames over the $q$-dimensional sphere; further details can be found
in \cite{npw1, npw2}, see also \cite{bkmpAoS, bkmpAoSb, gm1, gelpes, mpbb08}
and \cite{mp-book}, Chapter 10. From now on, we will use the simplified
notation $L^{2}\left( dz\right) =L^{2}\left( \mathbb{S}^{q},dz\right) $ to
denote the space of square-integrable functions with respect to Lebesgue
measure on the sphere. It is well-known result that the following
decomposition holds:%
\begin{equation*}
L^{2}\left( dz\right) =\oplus _{\ell =0}^{\infty }\left( \mathcal{H}_{\ell
}\right) ,
\end{equation*}%
where $\mathcal{H}_{\ell }$ is the restriction to $\mathbb{S}^{q}$ of the
homogeneous polynomials on $\mathbb{R}^{q+1}$, for which an orthonormal
basis is provided by the system of spherical harmonics $\left\{ Y_{\ell
,m}\right\} _{m=1,\ldots,d_{\ell ,q}}$ of degree $\ell $, with dimension 
\begin{equation*}
d_{\ell ,q}=\frac{\ell +\eta _{q}}{\eta _{q}}\binom{\ell +2\eta _{q}-1}{\ell 
},\eta _{q}=\left( q-1\right) /2.
\end{equation*}%
Given any $f\in L^{2}\left( dz\right) $, the orthogonal projector over $%
\mathcal{H}_{\ell }$ is provided in spherical coordinates by the kernel
operator%
\begin{equation*}
P_{\ell ,q}f\left( z\right) =\sum_{m=1}^{d_{\ell ,q}}a_{\ell m}Y_{\ell
,m}\left( z\right),\quad z\in \mathbb{S}^{q},\quad a_{\ell m}=\int_{\mathbb{S}^{q}}\overline{Y}_{\ell ,m}\left( z\right) f\left( z\right) dz,
\end{equation*}%
see for instance \cite{steinweiss}. For $z_{1},z_{2}\in \mathbb{S}^{q}$, the
kernel associated to the projector $P_{\ell ,q}$ is given by%
\begin{equation*}
P_{\ell ,q}\left( z_{1},z_{2}\right) =\sum_{m=1}^{d_{\ell ,q}}\overline{Y}%
_{\ell ,m}\left( z_{1}\right) Y_{\ell ,m}\left( z_{2}\right) =\frac{\ell
+\eta _{q}}{\eta _{q}\omega _{q}}\mathcal{C}_{\ell }^{\left( \eta
_{q}\right) }\left( \left\langle z_{1},z_{2}\right\rangle \right) ,
\end{equation*}%
where $\left\langle \cdot ,\cdot \right\rangle $ is the standard scalar
product over $\mathbb{R}^{q+1}$, $\mathcal{C}_{\ell }^{\left( \eta
_{q}\right) }$ denotes the Gegenbauer polynomial of degree $\ell $ with
parameter $\eta _{q}$, (see for instance \cite{szego}), and $\omega _{q}$ is
the measure of the surface of the $q$-dimensional sphere, namely%
\begin{equation*}
\omega _{q}=\frac{2\pi ^{\frac{q+1}{2}}}{\Gamma \left( \frac{q+1}{2}\right) }%
.
\end{equation*}%
We write $\mathcal{K}_{\ell }=\oplus _{i=0}^{\ell }\mathcal{H}_{i}$ for the
linear space of polynomials with degree smaller or equal than $\ell $; as
showed in \cite{npw1}, for every integer $\ell =1,2,\ldots$ there exists a
finite set of cubature points $\left\{ \xi \right\} \in \mathcal{Q}%
_{l}\subset \mathbb{S}^{q}$ and corresponding weights $\left\{ \lambda _{\xi
}\right\} $, such that, for $f$ $\in \mathcal{K}_{\ell }$%
\begin{equation*}
\int_{\mathbb{S}^{q}}f\left( x\right) dx=\sum_{\left\{ \xi \right\} \in 
\mathcal{Q}_{l}}\lambda _{\xi }f\left( \xi \right) .
\end{equation*}%
Now let us fix a parameter $B>1$; we will denote by $\left\{ \xi
_{jk}\right\} _{k=1,\ldots,K_{j}}=\mathcal{Q}_{\left[ 2B^{j+1}\right] }$, and $%
\left\{ \lambda _{jk}\right\} _{k=1,\ldots,K_{j}}$ the set of cubature points
and weights associated to the resolution level~$j$: we recall that $\lambda
_{jk}\approx B^{-qj}$ and $K_{j}\approx B^{qj}$, where $a\approx b$
indicates that there exist two positive constants $c_{1},c_{2}$ s.t. $%
c_{1}b^{-1}\leq a\leq cb$. Consider a real-valued function $b$ on $\left(
0,\infty \right) $, such that (i) $b(\cdot )$ has compact support in $\left[
B^{-1},B\right] $; (ii) $b\in C^{\infty }\left( \mathbb{S}^{q}\right) $;
(iii) for every $\ell \geq 1$ $\sum_{j=0}^{\infty }b^{2}\left( \ell
/B^{j}\right) =1$. The set of spherical needlets is then defined as 
\begin{equation}
\psi _{jk}(z)=\sqrt{\lambda _{jk}}\sum_{\ell }b\left( \frac{\ell }{B^{j}}%
\right) \frac{\ell +\eta _{q}}{\eta _{q}\omega _{q}}\mathcal{C}_{\ell
}^{\left( \eta _{q}\right) }\left( \left\langle z,\xi _{jk}\right\rangle
\right) ,z\in \mathbb{S}^{q}.  \label{sbatomattina}
\end{equation}%
The needlet coefficients (of index $j$, $k$) are defined as follows%
\begin{equation*}
\beta _{jk}:=\int_{\mathbb{S}^{q}}f\left( z\right) \psi _{jk}\left( z\right)
dz,
\end{equation*}%
and the following reconstruction formula holds, in the $L^{2}$ sense:%
\begin{equation*}
f\left( z\right) =\sum_{j,k}\beta _{jk}\psi _{jk}\left( z\right) ,%
z\in \mathbb{S}^{q}.
\end{equation*}%
The following localization property was established by \cite{npw1} (see also 
\cite{gm1}, \cite{gelpes}): for any positive integer $\tau $, there exists $%
\kappa _{\tau }>0$ such that for any $j,k$ and $z\in \mathbb{S}^{q}$%
\begin{equation}
\left\vert \psi _{jk}\left( z\right) \right\vert \leq \frac{\kappa _{\tau
}B^{\frac{q}{2}j}}{\left( 1+B^{\frac{q}{2}j}d\left( z,\xi _{jk}\right)
\right) ^{\tau }},  \label{localization}
\end{equation}%
where $d\left( \cdot ,\cdot \right) $ is the geodesic distance on the sphere
(i.e. for $q=2\ $and $z_{1},z_{2}\in \mathbb{S}^{q}$, $d\left(
z_{1},z_{2}\right) =\arccos \left( z_{1},z_{2}\right) $). Following \cite%
{npw2}, from this localization result, the following bounds on the $L^{p}$%
-norms hold:%
\begin{equation}
c_{p}B^{jq\left( \frac{p}{2}-1\right) }\leq \left\Vert \psi _{jk}\right\Vert
_{L^{p}\left( \mathbb{S}^{q}\right) }^{p}\leq C_{p}B^{jq\left( \frac{p}{2}%
-1\right) }.  \label{Lpbound}
\end{equation}%
\begin{remark}
\label{stilltobewritten} In the sequel, for $z_{1},z_{2}\in \mathbb{S}^{q}$,
we shall also meet functions $\psi _{j}^{\left( s\right) }\left( \cdot
,\cdot \right) $ given by 
\begin{equation}
\psi _{j}^{\left( s\right) }\left( z_{1},z_{2}\right) =B^{-\frac{q}{2}%
j}\sum_{\ell }b^{s}\left( \frac{\ell }{B^{j}}\right) \frac{\ell +\eta _{q}}{%
\eta _{q}\omega _{q}}\mathcal{C}_{\ell }^{\left( \eta _{q}\right) }\left(
\left\langle z_{1},z_{2}\right\rangle \right) .
\label{needletbehave}
\end{equation}%
It is immediate to see that for any integer $s>0$, the function $b^{s}\left(
\cdot \right) $ is compactly supported, nonnegative and it belongs \ to the
space $C^{\infty }\left( \mathbb{R}\right) $: therefore, following the same
arguments to establish the localization property in \cite{npw1} and \cite%
{gm2}, it can be shown that, for any $\tau >2$, there exist a constant $%
C_{\tau }>$ such that%
\begin{equation*}
\left\vert \psi _{j}^{\left( s\right) }\left( z_{1},z_{2}\right) \right\vert
\leq \frac{C_{\tau }B^{\frac{q}{2}j}}{\left( 1+B^{\frac{q}{2}j}d\left(
\left\langle z_{1},z_{2}\right\rangle \right) \right) ^{\tau }},
\end{equation*}%
and hence%
\begin{equation*}
\psi _{j}^{\left( s\right) }\left( z_{1},z_{2}\right) =O_{j}\left( B^{\frac{q%
}{2}j}\right) .
\end{equation*}
\end{remark}

\subsection{Statement of the main results}

\subsubsection{Poissonized case}
\noindent Throughout this paper, we shall assume that the function $f$ in the
governing Poisson measure is bounded and bounded away from zero, e.g. $m\leq f\left( z\right) \leq M$ for some $m,M>0$ for all $z\in 
\mathbb{S}^{q}$. Let us now consider first the vector of $U$-statistics%
\begin{equation}
U_{j}^{\left( 1\right) }\left( t\right) =\left( U_{jk_{1}}^{\left( 1\right)
}\left( t\right) ,\ldots,U_{jk_{d}}^{\left( 1\right) }\left( t\right) \right)
\label{multigeneral1}
\end{equation}%
where for any $k=k_{i},$ $i=1,\ldots,d$, we have 
\begin{equation}
U_{jk}^{\left( 1\right) }\left( t\right) =\frac{1}{u!}\sum_{(z_{1},z_{2})\in
N_{\neq }^{2}}(\psi _{jk}(z_{1})-\psi _{jk}(z_{2}))^{u},  \label{general1bis}
\end{equation}%
and the needlet functions $\psi _{jk_{i}}(\cdot )$ are given by \eqref%
{sbatomattina}, for some fixed locations $\left\{ \xi _{k_{i}}\right\} \in 
\mathbb{S}^{q}$, $i=1,\ldots,d$. Observe that \eqref{general1bis} has the form
of \eqref{general0} where, for $z_{1},z_{2}\in \mathbb{S}^{q}$, the kernel $%
h_{j}\left( \cdot ,\cdot \right) $ is defined as 
\begin{equation}
h_{j}(z_{1},z_{2})\equiv h_{jk;u}(z_{1},z_{2}):=\frac{1}{u!}(\psi
_{jk}(z_{1})-\psi _{jk}(z_{2}))^{u}.  \label{general2}
\end{equation}

\noindent In the case $u=2$, \eqref{general1bis} provides the sample variance of the
(de-Poissonized) random variables $\psi _{jk}\left( X\right) $ (up to a
normalization factor), where $X\in \mathbb{S}^{q}$ has density $f(\cdot )$,
while for $u=3$, it provides skewness estimator. More precisely, for $u=2$
it is a standard exercise to show that 
\begin{equation*}
\frac{1}{R_{t}^{2}}\mathbb{E}\left[ U_{jk}^{\left( 1\right) }\left( t\right) %
\right] =\int_{\mathbb{S}^{q}}\psi _{jk}^{2}\left( z\right) f\left( z\right)
dz-\left( \int_{\mathbb{S}^{q}}\psi _{jk}(z)f(z)dz\right) ^{2},
\end{equation*}%
i.e., $U_{jk}^{\left( 1\right) }\left( t\right) $ provides an unbiased
estimator for the variance of $\psi _{jk}\left( X\right) $.
\\~\\
\noindent As a second application, we shall consider so-called Sobolev tests of
uniformity, i.e. testing the null hypothesis that the function $f$ is
constant over the sphere (that is, $f\left( z\right) =\frac{1}{\omega _{q}},$
$z\in \mathbb{S}^{q}$), as discussed for instance by \cite{Jupp1}, \cite%
{Jupp2}, see also \cite{gine1}; in these references, the corresponding
statistics are built out of Fourier basis over manifolds, and in the case of
the sphere they would take the following form (compare \cite{Jupp2}, pp.1247
and following):

\begin{equation}
\sum_{\ell _{1}\ell _{2}}a_{\ell _{1}}a_{\ell _{2}}\left(
\sum_{(z_{1},z_{2})\in N_{\neq }^{2}}\sum_{m=1}^{d_{\ell ,q}}\overline{Y}%
_{\ell ,m}\left( z_{1}\right) Y_{\ell ,m}\left( z_{2}\right) \right).  \label{jupp}
\end{equation}%
Here, $\left\{ a_{\ell }\right\} $ is a square-summable sequence introduced
to combine the statistics evaluated at different multipoles $\ell $ into a
single value; actually the procedure discussed by \cite{Jupp2} is slightly
different as it includes in the sum also the diagonal terms $z_{1}=z_{2},$
but it is simple to show that after centering the two alternatives are
asymptotically equivalent. The integral of spherical harmonics with respect
to the uniform measure is obviously zero, so \eqref{jupp} provides a natural
statistic to test uniformity.
\\~\\
\noindent Our proposal exploits the same idea, with two modifications: we consider a
needlet-frame, rather than a Fourier dictionary, and we manage to provide
asymptotic behaviour also for the single summands, rather than for the
combined statistics. More precisely, we advocate the usage of following
vector of U-statistics%
\begin{equation*}
U_{j_{1},\ldots,j_{d}}^{\left( 2\right) }\left( t\right) =\left(
U_{j_{1}}^{\left( 2\right) }\left( t\right) ,\ldots,U_{j_{d}}^{\left( 2\right)
}\left( t\right) \right) ,
\end{equation*}%
where for any $j=j_{i},$ $i=1,\ldots,d$, we have%
\begin{equation}
U_{j}^{\left( 2\right) }\left( t\right) =\sum_{(z_{1},z_{2})\in N_{\neq
}^{2}}\sum_{k}\psi _{jk}\left( z_{1}\right) \psi _{jk}\left( z_{2}\right) 
.  \label{general2bis}
\end{equation}

\noindent Observe that \eqref{general2bis} has the form of \eqref{general0} where the
kernel $h_{j}\left( \cdot ,\cdot \right) $ is defined as 
\begin{equation*}
h_{j}(z_{1},z_{2})\equiv h_{j}(z_{1},z_{2}):=\sum_{k}\psi _{jk}\left(
z_{1}\right) \psi _{jk}\left( z_{2}\right) . 
\end{equation*}
As for the spherical harmonics, it is readily seen that under the null
hypothesis of uniformity, $f\left( z\right) =\omega _{q}^{-1}$, $z\in 
\mathbb{S}^{q}$, we have 
\begin{equation*}
\int_{\mathbb{S}^{q}}\psi _{jk}\left( z\right) f\left( z\right) dz=\frac{1}{%
\omega _{q}}\int_{\mathbb{S}^{q}}\psi _{jk}\left( z\right) dz=0;
\end{equation*}%
more generally, this integral is zero for all functions $f\left( \cdot
\right) $ that are band-limited, i.e. when $f\in \left\{ \oplus _{\ell
=0}^{B^{j^{\prime }}}\mathcal{H}_{\ell }\right\} ,$ $j^{\prime }<j-1$; in
this sense, needlets provide a natural building block to implement a Sobolev
test of uniformity, and investigating the full vector of statistics $%
U_{j_{1},\ldots,j_{d}}^{\left( 2\right) }\left( t\right) $ seems to provide
more information than combining the components into a single value.
\\~\\
\noindent As argued below, due to the real-domain localization properties of the
needlet frames both \eqref{general2} and \eqref{general2bis} are feasible
and asymptotically justifiable in circumstances where the sphere $\mathbb{S}%
^{q}$ is only partially observed, a situation which takes place very often
in practice. This means, for instance, that it may be feasible to test for
uniformity of the function $f\left( \cdot \right) $ even from observations
which cover a fraction of the sky, as it is the case for most astrophysical
experiments (\cite{iuppa}).
\\~\\
\noindent Before stating our main results, additional notation is needed. For any
given random vector $X=\left( X_{1},\ldots,X_{d}\right) $, we denote by $%
\widetilde{X}$ the normalized counterpart of $X$\ given by%
\begin{equation*}
\widetilde{X}:=\left( \frac{X_{1}-\mathbb{E}\left( X_{1}\right) }{\sqrt{{\rm Var}\left(
X_{1}\right) }},\ldots,\frac{X_{d}-\mathbb{E}\left( X_{d}\right) }{\sqrt{{\rm Var}\left(
X_{d}\right) }}\right) . 
\end{equation*}
Furthermore, let $Z_{d}$ be a centered $d$-dimensional Gaussian vector with
the identity definite covariance matrix. Our first result covers the
statistics defined in \eqref{general2}:

\begin{theorem}
\label{maintheorem1}Let $U_{j}^{\left( 1\right) }\left( t\right) $ be given
by \eqref{multigeneral1}. Then, for any $\tau >2$,%
\begin{equation*}
d_{2}\left( \widetilde{U_{j}^{\left( 1\right) }}\left( t\right)
,Z_{d}\right) =O\left( B^{\frac{q}{2}j}R_{t}^{-\frac{1}{2}}+B^{-\frac{q}{2}%
j\tau }+R_{t}^{-1}\right) .
\end{equation*}%
Therefore, taking $j=j\left( t\right) $ such that $j\left( t\right) \underset%
{t\rightarrow \infty }{\rightarrow }\infty $ and $B^{\frac{q}{2}j\left(
t\right) }R_{t}^{-\frac{1}{2}}=o_{t}\left( 1\right) $, it follows that%
\begin{equation*}
\widetilde{U_{j}^{\left( 1\right) }}\left( t\right) \rightarrow _{d}Z_{d}%
.
\end{equation*}
\end{theorem}

\noindent As a consequence of this theorem, we immediately obtain the following
corollary for $d=1$.

\begin{corollary}
\label{corol1}
For any given $f$ and for any $j,k$,%
\begin{equation*}
d_{W}\left( \widetilde{U_{jk}^{\left( 1\right) }}\left( t\right) ,\mathcal{N}%
\left( 0,1\right) \right) =O\left( B^{j\frac{q}{2}}R^{-\frac{1}{2}}\right)
\end{equation*}
and
\begin{equation*}
d_{K}\left( \widetilde{U_{jk}^{\left( 1\right) }}\left( t\right) ,\mathcal{N}%
\left( 0,1\right) \right) =O\left( B^{j\frac{q}{2}}R^{-\frac{1}{2}}\right).
\end{equation*}
\end{corollary}

\noindent Our second main result covers the statistics defined by \eqref{general2bis}:

\begin{theorem}
\label{maintheorem2} Let $U_{j_{1},\ldots,j_{d}}^{\left( 2\right) }\left(
t\right) =\left( U_{j_{1}}^{\left( 2\right) }\left( t\right)
,\ldots,U_{j_{d}}^{\left( 2\right) }\left( t\right) \right) $ be a $d$%
-dimensional vector where the components are of the form \eqref{general2bis}, with $\left\vert j_{i}-j_{i^{\prime }}\right\vert \geq 2$ for all $i\neq
i^{\prime }$. Then, there exists a constant $C_{d}>0$ such that 
\begin{equation*}
d_{2}\left( \widetilde{U_{j_{1},\ldots,j_{d}}^{\left( 2\right) }}\left(
t\right) ,Z_{d}\right) =O\left( \max_{j_{1},\ldots,j_{d}}\left( B^{\frac{q}{2}%
j_{i}}R_{t}^{-\frac{1}{2}}+B^{-\frac{q}{2}j_{i}}+R_{t}^{-\frac{1}{2}}\right) \right) 
.
\end{equation*}%
Therefore, for $i=1,\ldots,d$, taking $j_{i}=j_{i}\left( t\right) $ such that $%
j_{i}\left( t\right) \underset{t\rightarrow \infty }{\rightarrow }\infty $
and $B^{\frac{q}{2}j_{i}\left( t\right) }R_{t}^{-\frac{1}{2}}=o_{t}\left(
1\right) $, it follows that%
\begin{equation*}
\widetilde{U_{j_{1},\ldots,j_{d}}^{\left( 2\right) }}\left( t\right)
\rightarrow _{d}Z_{d}.
\end{equation*}
\end{theorem}
\noindent As a consequence of this theorem, we immediately obtain the following
corollary for $d=1$.
\begin{corollary}
\label{corol2}
For any given $f$ and for any $j$,
\begin{equation*}
d_{W}\left( \widetilde{U_{j}^{\left( 2\right) }}\left( t\right) ,\mathcal{N}%
\left( 0,1\right) \right) =O\left( B^{\frac{q}{2}j}R_{t}^{-\frac{1}{2}}\right)
\end{equation*}
and
\begin{equation*}
d_{K}\left( \widetilde{U_{j}^{\left( 2\right) }}\left( t\right) ,\mathcal{N}%
\left( 0,1\right) \right) =O\left( B^{\frac{3}{4}qj}R^{-\frac{3}{4}}\right).
\end{equation*}
\end{corollary}
\begin{remark}
In Corollaries \ref{corol1} and \ref{corol2}, the bound on the Kolmogorov distance is a direct consequence of \cite[Theorem 4.1]{et1} and the proofs of these two results are omitted for the sake of brevity. 
\end{remark}
\begin{remark}
The results in both the theorems can be easily generalized to cover the case
where the dimension $d$ grows itself with the "time" parameter $t$. The
details are completely analogous to those given in related circumstances by 
\cite{dmp2014}, and hence are omitted here for brevity's sake.
\end{remark}

\begin{remark}
The rates in Theorems \ref{maintheorem1} and \ref{maintheorem2} are actually
different and indeed the proofs of these results are based on unrelated
arguments. In particular, for Theorem \ref{maintheorem1}, we shall show that
the asymptotic behaviour is governed by a stochastic integral belonging to
the first Wiener chaos, while for Theorem \ref{maintheorem2}, the dominant
term is a double stochastic integral with respect to the underlying Poisson
random measure. The condition $B^{\frac{q}{2}j_{i}\left( t\right) }R_{t}^{-%
\frac{1}{2}}=o_{t}\left( 1\right) $ should be interpreted as the requirement
that "the effective sample size" diverges to infinity, as argued in related
circumstances by \cite{dmp2014}.
\end{remark}

\begin{remark}
The components of the vector $\widetilde{U_{j}^{\left( 1\right) }}\left(
t\right) $ in Theorem \ref{maintheorem1} are related to needlets evaluated
at the same scale $j,$ but around different locations $(\xi _{k_{1}},\ldots, \xi
_{k_{d}})$ on the sphere; on the other hand, the components of the vector $%
\widetilde{U_{j_{1},\ldots,j_{d}}^{\left( 2\right) }}\left( t\right) $ in
Theorem \ref{maintheorem2} are evaluated on the full sphere at different
frequencies $(j_{1},\ldots,j_{d}).$ Both results are formulated under the
assumption that the sphere is fully observable. This is done, however, only
for notational simplicity: as mentioned earlier, exploiting the localization
properties of the needlet construction it is simple modify the statements
for the case where these statistics are evaluated only on subsets of the
sphere, the same convergence rates in the quantitative central limit
theorems remaining valid up to constants. The arguments are completely
analogous to those exploited for instance in \cite{bkmpAoS}, Section 7 in a
Gaussian environment, and they are omitted here for brevity's sake.
\end{remark}

\begin{remark}
In Theorem \ref{maintheorem2} the assumption that $\left\vert
j_{i}-j_{i^{\prime }}\right\vert \geq 2$ ensures that the limiting
covariance matrix is exactly diagonal; relaxing this assumption makes the
statement notationally more complicated but does not require any new ideas
for the proofs.
\end{remark}

\begin{remark}
The expressions for the mean and variance in Theorem \ref{maintheorem2} can
be provided in a very explicit analytic form, for any fixed value of $j.$
Moreover we have also the asymptotic convergence, for $j,t\rightarrow \infty 
$%
\begin{equation*}
\frac{1}{B^{qj}R_{t}^{2}}\mathbb{E}\left[ \widetilde{U_{j}^{\left( 2\right) }%
}\left( t\right) \right] \rightarrow \gamma _{q} \quad \text{and} \quad \frac{1}{B^{qj}R_{t}^{2}}{\rm Var}\left[ \widetilde{U_{j}^{\left( 2\right) }}%
\left( t\right) \right] \rightarrow 2\gamma _{q},
\end{equation*}%
where%
\begin{equation}
\gamma _{q}:=\frac{1}{\eta _{q}\omega _{q}^{2}}\frac{1}{(q-2)!}%
\int_{1/B}^{B}b^{4}(u)u^{q-1}du.  \label{gammabumbum}
\end{equation}%
It should be noted that the asymptotic variance is exactly twice the
expected value, suggesting a natural interpretation of the distribution of $%
\left\{ \widetilde{U_{j}^{\left( 2\right) }}\left( t\right) \right\} $ in
terms of a generalized chi-square law with (real-valued) degrees of freedom
diverging to infinity.
\end{remark}

\subsubsection{A quantitative de-Poissonization Lemma}\label{ss:1}

\noindent In what follows, we shall show that the explicit bounds stated in Theorem \ref{maintheorem1} and Theorem \ref{maintheorem2} can be extended, at the cost of an additional factor, to the case of $U$-statistics based on a vector of i.i.d. observations, rather than on a Poisson measure. Our main tool in order to achieve this task is a new quantitative version of an argument taken from the fundamental paper by Dynkin and Mandelbaum \cite{DyMa}, that we shall state in the general framework of $U$-statistics of arbitrary order. Note that one could alternatively deal with the one-dimensional case by using general Berry-Esseen bounds for $U$-statistics (see e.g. \cite{bjz}); however we believe that our approach (which has independent interest) is more adapted in to directly study multi-dimensional probabilistic approximations. In the statement of the forthcoming Lemma \ref{l:dep}, we shall work within the following framework: 
\begin{itemize}

\item[--] $X = \{X_i : i\geq 1\}$ is a sequence of i.i.d. random variables with values in some measurable space $(E, \mathcal{E})$;

\item[--] Let $m\geq 1$ be a fixed integer: we write $\{h_n : n\geq 1\}$ to indicate a sequence of jointly measurable symmetric kernels $h_n : E^m\to \R$ such that $\mathbb{E}[h_n(X_1,...,X_m)] = 0$ and $\mathbb{E}[h_n(X_1,...,X_m)^2]<\infty$;

\item[--] $\{N_n : n\geq 1\}$ is a sequence of Poisson random variables independent of $X$, such that $N(n)$ has a Poisson distribution with mean $n$ for every $n$;

\item[--] For every $n\geq m$, the symbol $U_n$ denotes the Poissonized $U$-statistic
$$
U_n = \sum_{1\leq i_1,..., i_m\leq N(n)} h_n(X_{i_1},..., X_{i_m}),
$$
where the sum runs over al $m$-ples $(i_1,...,i_m)$ such that $i_j\neq i_k$ for $j\neq k$.
\item[--] For every $n\geq 1$, the symbol $U'_n$ denotes the classical $U$-statistic
$$
U'_n = \sum_{1\leq i_1,..., i_m\leq n} h_n(X_{i_1},..., X_{i_m}),
$$
where, as before, the sum runs over al $m$-ples $(i_1,...,i_m)$ such that $i_j\neq i_k$ for $j\neq k$. It is easily checked that $\mathbb{E}[U_n] = \mathbb{E}[U'_n] = 0$.

\end{itemize}

\begin{lemma}[Quantitative de-Poissonization lemma]\label{l:dep} Assume that $\mathbb{E}[U_n^2]\to 1$ for every $n$. Then $\mathbb{E}[U_n'^2] \to 1$, and
$$
\mathbb{E}[(U_n - U'_n)^2] = O(n^{-1/2}), \quad n\to\infty.
$$
\end{lemma}

\subsubsection{Applications to needlet-based $U$-statistics of i.i.d. observations}\label{ss:2}

\noindent A direct application of Lemma \ref{l:dep} leads to de-Poissonized versions of the main results of the paper. In what follows, we shall denote by $\{X_i : i\geq 1\}$ a sequence of i.i.d. random variables with values in $\mathbb{S}^q$, whose common distribution has a density $f$ with respect to the Lebesgue measure. As before, we assume that $0<m\leq f(z) \leq M<\infty$, for every $z\in \mathbb{S}^q$. 
\\~\\
We start with an extension of Theorem \ref{maintheorem1}. For every $n$, we write  
$$U_{j}^{\left( 1\right) }\left( n\right)' = \left(U_{jk_1}^{(1)}(n)',...,U_{jk_d}^{(1)}(n)'\right) $$ 
where, for $k=k_i$, $i=1,...,d$,
$$
U_{jk}^{(1)}(n)' = \frac{1}{u!} \sum_{1\leq i_1\neq i_2 \leq n} (\psi_{jk} (X_{i_1}) - \psi_{jk}(X_{i_2}))^u, 
$$
that is, each $U_{jk}^{(1)}(n)'$ is obtained from \eqref{general1bis} (in the case $t=n$) by replacing the Poisson measure on the sphere $A\mapsto N([0,n]\times A)$ with the random measure $A\mapsto \sum_{i=1}^n \delta_{X_i}(A)$, where $\delta_x$ stands for the Dirac mass at $x$.

\begin{theorem}
\label{maintheorem3}  Under the above notation and assumptions,  for any $\tau >2$,%
\begin{equation*}
d_{2}\left( \widetilde{U_{j}^{\left( 1\right) }}\left( n\right)'
,Z_{d}\right) =O\left( B^{\frac{q}{2}j}n^{-\frac{1}{2}}+B^{-\frac{q}{2}%
j\tau }+n^{-1}+n^{-1/4}\right) .
\end{equation*}%
Therefore, taking $j=j\left( n\right) $ such that $j\left( n\right) \underset%
{t\rightarrow \infty }{\rightarrow }\infty $ and $B^{\frac{q}{2}j\left(
n\right) }n^{-\frac{1}{2}}=o_{n}\left( 1\right) $, it follows that%
\begin{equation*}
\widetilde{U_{j}^{\left( 1\right) }}\left( n\right)' \rightarrow _{d}Z_{d}%
.
\end{equation*}
\end{theorem}
\noindent Analogously to the notation introduced above, we shall write, for every $n$,
$$U_{j_1,...,j_d}^{\left( 2\right) }\left( n\right)' = \left(U_{j_1}^{(2)}(n)',...,U_{j_d}^{(2)}(n)'\right) $$ 
where, for $j=j_a$, $a=1,...,d$,
$$
U_{j}^{(2)}(n)' = \sum_{1\leq i_1\neq i_2 \leq n} \sum_k \psi_{jk} (X_{i_1})  \psi_{jk}(X_{i_2}). 
$$

\noindent The following result extends Theorem \ref{maintheorem3}.

\begin{theorem}
\label{maintheorem4} Under the above notation and assumptions, assume in addition that $\left\vert j_{i}-j_{i^{\prime }}\right\vert \geq 2$ for all $i\neq
i^{\prime }$. Then, there exists a constant $C_{d}>0$ such that 
\begin{equation*}
d_{2}\left( \widetilde{U_{j_{1},\ldots,j_{d}}^{\left( 2\right) }}\left(
n\right)' ,Z_{d}\right) =O\left( \max_{j_{1},\ldots,j_{d}}\left( B^{\frac{q}{2}%
j_{i}}n^{-\frac{1}{2}}+B^{-\frac{q}{2}j_{i}}+n^{-\frac{1}{2}} + n^{-1/4}\right) \right) 
.
\end{equation*}%
Therefore, for $i=1,\ldots,d$, taking $j_{i}=j_{i}\left( n\right) $ such that $%
j_{i}\left( n\right) \underset{t\rightarrow \infty }{\rightarrow }\infty $
and $B^{\frac{q}{2}j_{i}\left( n\right) }n^{-\frac{1}{2}}=o_{n}\left(
1\right) $, it follows that%
\begin{equation*}
\widetilde{U_{j_{1},\ldots,j_{d}}^{\left( 2\right) }}\left( n\right)'
\rightarrow _{d}Z_{d}.
\end{equation*}
\end{theorem}
\begin{remark}
The presence of the additional term $n^{-1/4}$ in the bound of Theorems \ref{maintheorem4} and \ref{maintheorem3} yields a phase transition in the convergence to the normal distribution. Indeed, depending on how fast $j(n)$ grows to infinity, the rate of convergence could be given either by $n^{-1/4}$ or by $B^{-\frac{q}{2}j(n)}$, with an equivalence between those two rates when $j(n) = \frac{\log(n)}{2q\log(B)}$. Let us write $$s_1(n) := \frac{q j(n)\log(B)}{2\log(n)} - \frac{1}{4} \quad \text{and} \quad s_2(n) := \frac{q\tau j(n)\log(B)}{2\log(n)} - \frac{1}{4}.$$ Then, the rate of convergence in Theorem \ref{maintheorem3} will be given by $B^{-\frac{q}{2}j(n)\tau}$ whenever $s_1(n) = o(\log(n))$ and by $n^{-1/4}$ otherwise. Similarly, the rate of convergence in Theorem \ref{maintheorem4} will be given by $B^{-\frac{q}{2}j(n)}$ whenever $s_2(n) = o(\log(n))$ and by $n^{-1/4}$ otherwise.
\end{remark}

\subsection{Plan of the paper}

\noindent The plan of this paper is as follows: the proofs for Theorems \ref%
{maintheorem1}, \ref{maintheorem2} are collected in Sections \ref{section2} and \ref{section3} respectively. Section \ref{s:lemma} deals with the proof of Lemma \ref{l:dep}. Auxiliary results are collected in Section \ref{appendix}, which
is divided into four parts, the first devoted to background results on
Stein-Malliavin approximations in a Poisson environment, the second
concerned with some functional inequalities for needlet kernels, the third
and fourth devoted to specific computations for the two main Theorems.

\section{Proof of Theorem \ref{maintheorem1}}
\label{section2}

\noindent Let us define $G_{n}\left( j\right) :=\int_{\mathbb{S}^{q}}\psi _{jk}^{n}\left( z\right)
f\left( z\right) dz$ and note that, from \eqref{Lpbound}
\begin{equation}
\left\vert G_{n}\left( j\right) \right\vert =O\left( B^{jq\left( \frac{n}{2}%
-1\right) }\right) .  \label{G_bound}
\end{equation}
\begin{example}
Assuming that $f$ is constant, $G_{1}\left( j\right) =0$ and $G_{2}\left(
j\right) =\left\Vert \psi _{jk}\right\Vert _{L^{2}\left( \mu _{t}\right)
}^{2}$. ~\newline
\end{example}
\noindent Recall that we are considering the process $U_{j}^{\left( 1\right) }\left(
t\right) =\left( U_{jk_{1}}^{\left( 1\right) }\left( t\right)
,\ldots,U_{jk_{d}}^{\left( 1\right) }\left( t\right) \right) $, whose
components, for $k=k_{1},\ldots,k_{d}$, are given by%
\begin{equation*}
U_{jk}^{\left( 1\right) }\left( t\right) =\frac{1}{u!}\sum_{(z_{1},z_{2})\in
N_{\neq }^{2}}(\psi _{jk}(z_{1})-\psi _{jk}(z_{2}))^{u}.
\end{equation*}%
From Lemma \ref{lemmagtogamma}, we have%
\begin{equation*}
\mathbb{E}\left[ U_{jk}^{\left( 1\right) }\left( t\right) \right]
= R_{t}^{2}\Gamma _{1}\left( j\right) \quad \text{and} \quad
{\rm Var}\left[ U_{jk}^{\left( 1\right) }\left( t\right) \right]
= R_{t}^{3}\Gamma _{21}\left( j\right) +R_{t}^{2}\Gamma _{22}\left(
j\right) ,
\end{equation*}%
where%
\begin{eqnarray}
\Gamma _{1}(j) &=&\sum_{r=0}^{u}\binom{u}{r}G_{u-r}\left( j\right)
G_{r}\left( j\right) ,  \label{gamma1} \\
\Gamma _{21}(j) &=&\sum_{q,r=0}^{u}\binom{u}{q}\binom{u}{r}%
G_{q}(j)G_{r}(j)G_{2u-\left( q+r\right) }\left( j\right),  \label{gamma2} \\
\Gamma _{22}(j) &=&\sum_{q,r=0}^{u}\binom{u}{q}\binom{u}{r}G_{q+r}\left(
j\right) G_{2u-\left( q+r\right) }\left( j\right) .  \label{gamma3}
\end{eqnarray}
\begin{remark}
Note that $\Gamma _{2i}(j)$, $i=1,2$, provides the variance of the
components of order $i$ in the Wiener chaos decomposition of $U_{jk}^{\left(
1\right) }\left( t\right) $ (see Lemma \ref{lemmagtogamma}).
\end{remark}
\begin{remark}
Note in particular that using the notation for (compensated) Poisson random
measure introduced in Subsection \ref{appendixsteinmall}, for $u=2$ we have that%
\begin{eqnarray*}
&& \frac{1}{2}\sum_{z_{1},z_{2}}\left( \psi _{jk}\left( z_{1}\right) -\psi
_{jk}\left( z_{2}\right) \right) ^{2} = \frac{1}{2}\int_{\left( \mathbb{S}%
^{q}\right) ^{2}}\left( \psi _{jk}\left( z_{1}\right) -\psi _{jk}\left(
z_{2}\right) \right) ^{2}N_{t}\left( dz_{1}\right) N_{t}\left( dz_{2}\right) \\
&& = \frac{1}{2}\int_{\left( \mathbb{S}^{q}\right) ^{2}}\left( \psi _{jk}\left(
z_{1}\right) -\psi _{jk}\left( z_{2}\right) \right) ^{2}\widehat{N}%
_{t}\left( dz_{1}\right) \widehat{N}_{t}\left( dz_{2}\right)  +\int_{\left( \mathbb{S}^{q}\right) ^{2}}\left( \psi _{jk}\left(
z_{1}\right) -\psi _{jk}\left( z_{2}\right) \right) ^{2}\mu _{t}\left(
dz_{1}\right) \widehat{N}_{t}\left( dz_{2}\right) \\
&& +\frac{1}{2}\int_{\left( \mathbb{S}^{q}\right) ^{2}}\left( \psi _{jk}\left(
z_{1}\right) -\psi _{jk}\left( z_{2}\right) \right) ^{2}\mu _{t}\left(
dz_{1}\right) \mu _{t}\left( dz_{2}\right)
\end{eqnarray*}%
so that%
\begin{equation*}
\mathbb{E}\left[ \frac{1}{2}\sum_{z_{1},z_{2}}\left( \psi _{jk}\left(
z_{1}\right) -\psi _{jk}\left( z_{2}\right) \right) ^{2}\right] = R_{t}^{2}\left[ \int_{\mathbb{S}^{q}}\psi _{jk}\left( z\right)
^{2}f\left( z\right) dz-\left( \int_{\mathbb{S}^{q}}\psi _{jk}\left(
z\right) f\left( z\right) dz\right) ^{2}\right] .
\end{equation*}
\end{remark}
\noindent Using Lemma \ref{Lemma_gamma}, we have that there exists a constant $\sigma
>0$ such that 
\begin{equation*}
\sigma :=\lim_{j\rightarrow \infty }\Gamma _{21}\left( j\right) B^{-j\frac{q%
}{2}\left( u-1\right) }.
\end{equation*}%
Therefore, from now we will consider the centred and asymptotically
normalized counterpart of $U_{jk}^{\left( 1\right) }\left( t\right) $, given
by 
\begin{equation*}
\widetilde{U_{jk}^{\left( 1\right) }}\left( t\right) =I_{2}\left( \widetilde{%
h}_{jk,t}\right) +I_{1}\left( \widetilde{g}_{jk,t}\right),
\end{equation*}
where
\begin{eqnarray*}
L^{2}\left( \mu _{t}^{\otimes 2}\right) &\ni &\widetilde{h}_{jk,t}\left(
z_{1},z_{2}\right) :=\left( z_{1},z_{2}\right) \mapsto \frac{\sum_{r=0}^{u}%
\binom{u}{r}\psi _{jk}^{u-r}\left( x\right) \psi _{jk}^{r}\left( y\right) }{%
R_{t}^{3/2}B^{jq\left( u-1\right) /2}\sigma },   \\
L^{2}\left( \mu _{t}\right) &\ni &\widetilde{g}_{jk,t}\left( z\right)
:=z\mapsto \frac{\int_{\mathbb{S}^{q}}\sum_{r=0}^{u}\binom{u}{r}\psi
_{jk}^{u-r}\left( z\right) \psi _{jk}^{r}\left( y\right) \mu _{t}\left(
dy\right) }{R_{t}^{3/2}B^{jq\left( u-1\right) /2}\sigma }.
\end{eqnarray*}
Let $Z_{d}\sim \mathcal{N}_{d}\left( 0,I\right) $ be a centered standard $d$--dimensional Gaussian vector and consider the following random vector 
\begin{equation*}
\widetilde{U_{j}^{\left( 1\right) }}\left( t\right) =\left( \widetilde{%
U_{jk_{1}}^{\left( 1\right) }}\left( t\right) ,\ldots,\widetilde{%
U_{jk_{d}}^{\left( 1\right) }}\left( t\right) \right) =\left( I_{1}\left( \widetilde{g}_{jk_{1},t}\right) +I_{2}\left( \widetilde{h%
}_{jk_{1,t}}\right) ,\ldots,I_{1}\left( \widetilde{g}_{jk_{d},t}\right)
+I_{2}\left( \widetilde{h}_{jk_{d},t}\right) \right) .
\end{equation*}
Our strategy to prove Theorem \ref{maintheorem1} will be based on two steps:
we shall bound the distance between the $U$-statistics we consider and an
approximating stochastic integral in the first Wiener chaos, and then bound
the probability distance between the latter and the limiting Gaussian
distribution. In particular, we shall focus on the distance%
\begin{equation*}
d_{2}\left( \widetilde{U_{j}^{\left( 1\right) }}\left( t\right)
,Z_{d}\right) =d_{2}\left( I_{1,j}\left( t\right) +I_{2,j}\left( t\right)
,Z_{d}\right) ,
\end{equation*}%
where%
\begin{equation*}
I_{1,j}\left( t\right) =\left( I_{1}\left( \widetilde{g}_{jk_{1},t}\right)
,\ldots,I_{1}\left( \widetilde{g}_{jk_{d},t}\right) \right) \quad \text{and} \quad I_{2,j}\left( t\right) = \left( I_{2}\left( \widetilde{h}_{jk_{1},t}\right)
,\ldots,I_{2}\left( \widetilde{h}_{jk_{d},t}\right) \right) .
\end{equation*}%
Applying the triangle inequality, we obtain%
\begin{eqnarray*}
d_{2}\left( \widetilde{U_{j}^{\left( 1\right) }}\left( t\right)
,Z_{d}\right) &\leq &d_{2}\left( \widetilde{U_{j}^{\left( 1\right) }}\left(
t\right) ,I_{1,j}\left( t\right) \right) +d_{2}\left( I_{1,j}\left( t\right)
,Z_{d}\right) \\
&\leq &\mathbb{E}\left[ \left\Vert \widetilde{U_{j}^{\left( 1\right) }}%
\left( t\right) -I_{1,j}\left( t\right) \right\Vert _{\mathbb{R}^{d}}^{2}%
\right] +d_{2}\left( I_{1,j}\left( t\right) ,Z_{d}\right) = \mathbb{E}\left[ \left\Vert I_{2,j}\left( t\right) \right\Vert _{\mathbb{R%
}^{d}}^{2}\right] +d_{2}\left( I_{1,j}\left( t\right) ,Z_{d}\right) .
\end{eqnarray*}%
Our task is then to study the asymptotic behaviour of these two summands. Consider now the $d$-dimensional random vector $\widetilde{U_{j}^{\left(
1\right) }}\left( t\right) $, where $d\leq K_{j}$ is fixed: following
Proposition \ref{prop_cov}, it holds that 
\begin{equation*}
\lim_{t\rightarrow \infty }\mathbb{E}\left[ I_{1}\left( g_{jk_{1},t}\left(
z\right) \right) ,I_{1}\left( g_{jk_{2},t}\left( z\right) \right) \right]
=\delta _{k_{1}}^{k_{2}}.
\end{equation*}%
We shall show that 
\begin{equation*}
d_{2}\left( \widetilde{U_{j}^{\left( 1\right) }}\left( t\right)
,Z_{d}\right) \leq \left( \frac{4dC_{2u}M}{R_{t}\sigma ^{2}}+\frac{dC_{\tau
}MB^{-\frac{q}{2}j\tau }}{mc_{2}^{2}\left( 1+\inf_{k_{1}\neq k_{2}}d\left(
\xi _{jk_{1}},\xi _{jk_{2}}\right) \right) ^{2\tau }}+d\frac{\sqrt{2\pi }}{8}%
\frac{C_{3u}M^{4}}{c_{2}^{3}m}\frac{B^{\frac{q}{2}j}}{R_{t}^{\frac{1}{2}}}%
\right) .
\end{equation*}%
Indeed,
\begin{eqnarray*}
\left\Vert I_{d}-\Sigma _{j,t}\right\Vert _{H.S.} &\leq &\sqrt{%
\sum_{k_{1}\neq k_{2}=1}^{d}\mathbb{E}^{2}\left[ I_{1}\left( \widetilde{g}%
_{jk_{1},t}\left( z\right) \right) ,I_{1}\left( \widetilde{g}%
_{jk_{2},t}\left( z\right) \right) \right] } \leq d\sup_{k_{1}\neq k_{2}=1,\ldots ,d}\frac{C_{\tau
,M,u,\sigma }}{\sigma ^{2}\left( 1+B^{\frac{q}{2}j}d\left( \xi _{jk_{1}},\xi
_{jk_{2}}\right) \right) ^{u\tau }} \\
&\leq & d\sup_{k_{1}\neq k_{2}=1,\ldots,d}\frac{1}{mc_{2}^{2}}\frac{C_{\tau
,M,u,\sigma }}{\left( 1+\inf_{k_{1}\neq k_{2}=1,\ldots,d}B^{\frac{q}{2}j}d\left(
\xi _{ju},\xi _{jv}\right) \right) ^{u\tau }}:= A_{1}\left( t\right) .
\end{eqnarray*}%
On the other hand 
\begin{equation*}
\mathbb{E}\left[ \left\Vert I_{2,j}\left( t\right) \right\Vert _{\mathbb{R}%
^{d}}^{2}\right] \leq \sum_{k=1}^{d}\left\Vert \widetilde{h}%
_{jk,t}\right\Vert _{L^{2}\left( \mu _{t}^{\otimes 2}\right) }^{2}\leq \frac{4C_{2u}Md}{R_{t}\sigma ^{2}}\left( 1+o\left( \frac{1}{R_{t}}%
\right) \right) := A_{2}\left( t\right) .
\end{equation*}%
Finally,
\begin{eqnarray*}
d_{2}\left( \widetilde{U_{j}^{\left( 1\right) }}\left( t\right) ,Z_{d}\right) &\leq &A_{1}\left( t\right) +A_{2}\left( t\right) +\frac{\sqrt{2\pi }}{8}%
\sum_{k_{1},k_{2},k_{3}=1}^{d}\int_{\mathbb{S}^{q}}\left\vert \widetilde{g}%
_{jk_{1},t}\left( z\right) \right\vert \left\vert \widetilde{g}%
_{jk_{2},t}\left( z\right) \right\vert \left\vert \widetilde{g}%
_{jk_{3},t}\left( z\right) \right\vert \mu _{t}\left( dz\right) \\
&\leq &A_{1}\left( t\right) +A_{2}\left( t\right) +\frac{\sqrt{2\pi }}{8}%
\frac{M}{R_{t}^{\frac{7}{2}}B^{\frac{3q}{2}\left( u-1\right) j}\sigma ^{3}}%
\sum_{k_{1},k_{2},k_{3}=1}^{d}\int_{\mathbb{S}^{q}}\left\vert \psi
_{jk_{1}}^{u}\left( z\right) \right\vert \left\vert \psi _{jk_{2}}^{u}\left(
z\right) \right\vert \left\vert \psi _{jk_{3}}^{u}\left( z\right)
\right\vert dz.
\end{eqnarray*}
Following Lemma \ref{lemma3needlets}, we have 
\begin{equation*}
d_{2}\left( \widetilde{U_{j}^{\left( 1\right) }}\left( t\right)
,X_{d}\right) \leq A_{1}\left( t\right) +A_{2}\left( t\right) +d\frac{\sqrt{%
2\pi }}{8}\frac{C_{3u}M^{4}}{c_{2}^{3}m}\frac{B^{\frac{q}{2}j}}{R_{t}^{\frac{%
1}{2}}},
\end{equation*}%
as claimed.
\section{Proof of Theorem \ref{maintheorem2}}
\label{section3}
\noindent In this section, our purpose is to study the statistic%
\begin{equation*}
U_{j_{1},\ldots,j_{d}}^{\left( 2\right) }\left( t\right) =\left(
U_{j_{1}}^{\left( 2\right) }\left( t\right) ,\ldots,U_{j_{d}}^{\left( 2\right)
}\left( t\right) \right)
\end{equation*}%
where for any $j=j_{i}$, $i=1,\ldots,d$,
\begin{equation*}
U_{j}^{\left( 2\right) }\left( t\right) =\sum_{(z_{1},z_{2})\in N_{\neq
}^{2}}\sum_{k}\psi _{jk}\left( z_{1}\right) \psi _{jk}\left( z_{2}\right) 
.
\end{equation*}%
Recall that here we are focussing on Sobolev tests of uniformity, and hence
we are assuming the Poisson governing measure is given by 
\begin{equation}
\mathbb{E}\left[ N_{t}\left( dz\right) \right] =\mu _{t}\left( dz\right)
=R_{t}f\left( z\right) dz=\frac{R_{t}}{\omega _{q}}dz.
\label{uniformitynonzero}
\end{equation}%
This yields 
\begin{equation*}
\int_{\mathbb{S}^{q}}\psi _{j_{i}k}\left( z\right) \mu _{t}\left( dz\right) =%
\frac{R_{t}}{\omega _{q}}\int_{\mathbb{S}^{q}}\psi _{j_{i}k}\left( z\right)
dz=0. 
\end{equation*}
Using this fact along with Proposition \ref{p:zap}, we have 
\begin{equation*}
U_{j_{i}}^{\left( 2\right) }\left( t\right) =I_{2}\left( \sum_{k}\psi
_{j_{i}k}\otimes \psi _{j_{i}k}\right) +\left\Vert \sum_{k}\psi
_{j_{i}k}\otimes \psi _{j_{i}k}\right\Vert _{L^{2}\left( \mu _{t}^{\otimes
2}\right) }^{2}.
\end{equation*}%
Let $\gamma _{j,q}$ be given by 
\begin{equation}
\gamma _{j,q}:=\left( \omega _{q}B^{qj}\right) ^{-1}\sum_{\ell
=B^{j-1}}^{B^{j+1}}b^{4}\left( \frac{\ell }{B^{j}}\right) \frac{\ell +\eta
_{q}}{\eta _{q}\omega _{q}}\binom{\ell +q-2}{\ell }.
\label{gammabumbumconq}
\end{equation}%
From the standard zero mean property of Wiener--It\^o integrals, we have that,
for any $j=j_{i}$, $i=1,\ldots,d$, 
\begin{equation*}
\mathbb{E}\left[ U_{j}^{\left( 2\right) }\left( t\right) \right] =\left\Vert
\sum_{k}\psi _{jk}\otimes \psi _{jk}\right\Vert _{L^{2}\left( \mu
_{t}^{\otimes 2}\right) }^{2}=R_{t}^{2}B^{qj}\gamma _{j,q},
\end{equation*}%
and again from the properties of stochastic integrals w.r.t. Poisson random
measure%
\begin{equation}
{\rm Var}\left( U_{j}^{\left( 2\right) }\left( t\right) \right) = \mathbb{E}\left[
I_{2}^{2}\left( \sum_{k}\psi _{jk}\otimes \psi _{jk}\right) \right]  = 2\left\Vert \sum_{k}\psi _{jk}\otimes \psi _{jk}\right\Vert _{L^{2}\left(
\mu _{t}^{\otimes 2}\right) }^{2}.  \label{summands}
\end{equation}%
Hence, using Lemmas \ref{corverysimilar} and \ref{lemmaverysimilar}, we
obtain%
\begin{equation*}
{\rm Var}\left\{ U_{j}^{\left( 2\right) }\left( t\right) \right\}
=2R_{t}^{2}B^{qj}\gamma _{j,q}. 
\end{equation*}
We can hence focus on the normalized statistics 
\begin{equation*}
\widetilde{U_{j}^{\left( 2\right) }}\left( t\right) :=\frac{U_{j}^{\left(
2\right) }\left( t\right) -R_{t}^{2}B^{qj}\gamma _{j,q}}{R_{t}B^{\frac{q}{2}%
j}\sqrt{2\gamma _{j,q}}}=I_{2}\left( \widetilde{h}_{j_{i},t}\right) ,
\end{equation*}%
where 
\begin{equation*}
L^{2}\left( \mu _{t}^{\otimes 2}\right) \ni \widetilde{h}_{j,t}\colon \left(
z_{1},z_{2}\right) \mapsto \frac{\sum_{k}\psi _{jk}\left( z_{1}\right) \psi
_{jk}\left( z_{2}\right) }{R_{t}B^{\frac{q}{2}j}\sqrt{2\gamma _{j,q}}}.
\end{equation*}%
Observe that the variance of $\widetilde{U_{j}^{\left( 2\right) }}\left(
t\right) $ is identically equal to $1$ as $j$ grows to $\infty $. Therefore,
in order to prove Theorem \ref{maintheorem2}, it remains to check that $%
\widetilde{h}_{j,t}$ satisfies the five conditions in Proposition \ref%
{Peccati}, for all $j=j_{i},i=1,\ldots,d$. Before doing so, notice that the
kernel $\widetilde{h}_{j,t}$ can be rewritten as 
\begin{equation*}
\widetilde{h}_{j,t}(z_{1},z_{2})=\left( R_{t}B^{\frac{q}{2}j}\sqrt{2\gamma
_{j,q}}\right) ^{-1}\sum_{\ell }b^{2}\left( \frac{\ell }{B^{j}}\right) \frac{%
\ell +\eta _{q}}{\eta _{q}\omega _{q}}\mathcal{C}_{\ell }^{\left( \eta
_{q}\right) }\left( \left\langle z_{1},z_{2}\right\rangle \right) ,
\end{equation*}%
as pointed out in Lemma \ref{rewritekernel}. Condition\textit{\ }1\textit{\ }%
in Proposition \ref{Peccati} is hence automatically satisfied by
construction.
\\~\\
For Condition 2\textit{, }following Lemma\textit{\ }\ref{lemmawithsinteger}%
\textit{, }we obtain\textit{\ }%
\begin{eqnarray*}
\left\Vert \widetilde{h}_{j,t}\right\Vert _{L^{4}\left( \mu _{t}^{\otimes
2}\right) }^{4} &=& \left( R_{t}B^{\frac{q}{2}j}\sqrt{2\gamma _{j,q}}\right)
^{-4}B^{2qj_{i}}\int_{\mathbb{S}^{q}}\left( \sum_{\ell }b^{2}\left( \frac{%
\ell }{B^{j}}\right) \frac{\ell +\eta _{q}}{\eta _{q}\omega _{q}}\mathcal{C}%
_{\ell }^{\left( \eta _{q}\right) }\left( \left\langle
z_{1},z_{2}\right\rangle \right) B^{-\frac{q}{2}j}\right) ^{4}\mu
_{t}^{\otimes 2}\left( dz_{1},dz_{2}\right) \\
&=& O\left( R_{t}^{-2}B^{2qj}\right) .
\end{eqnarray*}%
Now for Condition 3, as
\begin{eqnarray*}
\left( \widetilde{h}_{j,t}\star _{1}^{1}\widetilde{h}_{j,t}\right) \left(
z_{1},z_{2}\right)  &=&\left( R_{t}B^{\frac{q}{2}j}\sqrt{2\gamma _{j,q}}\right) ^{-2}R_{t}\int_{%
\mathbb{S}^{q}}\sum_{\ell _{1},\ell _{2}}\left( \prod_{i=1,2}b^{2}\left( 
\frac{\ell _{i}}{B^{j}}\right) \frac{\ell _{i}+\eta _{q}}{\eta _{q}\omega
_{q}}\mathcal{C}_{\ell _{i}}^{\left( \eta _{q}\right) }\left( \left\langle
z_{i},a\right\rangle \right) \right) da \\
&=&\frac{1}{R_{t}}\left( \sqrt{2\gamma _{j,q}}B^{\frac{q}{2}j}\right)
^{-2}\sum_{\ell }b^{4}\left( \frac{\ell }{B^{j_{i}}}\right) \frac{\ell +\eta
_{q}}{\eta _{q}\omega _{q}}\mathcal{C}_{\ell }^{\left( \eta _{q}\right)
}\left( \left\langle z_{1},z_{2}\right\rangle \right),
\end{eqnarray*}%
we obtain, using Lemma \ref{lemmawithsinteger} again,
\begin{eqnarray*}
\left\Vert \widetilde{h}_{j,t}\star _{1}^{1}\widetilde{h}_{j,t}\right\Vert
_{L^{2}\left( \mu _{t}^{\otimes 2}\right) }^{2} &=& \left( R_{t}B^{\frac{q}{2}j}\sqrt{2\gamma _{j,q}}\right)
^{-4}B^{qj}\int_{\left( \mathbb{S}^{q}\right) ^{2}}\left( \sum_{\ell
}b^{4}\left( \frac{\ell }{B^{j}}\right) \frac{\ell +\eta _{q}}{\eta
_{q}\omega _{q}}\mathcal{C}_{\ell }^{\left( \eta _{q}\right) }\left(
\left\langle z_{1},z_{2}\right\rangle \right) B^{-\frac{q}{2}j}\right)
^{2}dz_{1}dz_{2} \\
&=& O\left( B^{-qj}\right).
\end{eqnarray*}%
For Condition 4, we start by observing that
\begin{eqnarray*}
\left( \widetilde{h}_{j,t}\star _{2}^{1}\widetilde{h}_{j,t}\right) \left(
z\right) &=& R_{t}\left( R_{t}B^{\frac{q}{2}j}\sqrt{2\gamma _{j,q}}\right) ^{-2}\int_{%
\mathbb{S}^{q}}\sum_{\ell _{1},\ell _{2}}\left( \prod_{i=1,2}b^{2}\left( 
\frac{\ell _{i}}{B^{j}}\right) \frac{\ell _{i}+\eta _{q}}{\eta _{q}\omega
_{q}}\mathcal{C}_{\ell _{i}}^{\left( \eta _{q}\right) }\left( \left\langle
z,a\right\rangle \right) \right) da \\
&=& R_{t}^{-1}\left( B^{\frac{q}{2}j}\sqrt{2\gamma _{j,q}}\right)
^{-2}\sum_{\ell =B^{j-1}}^{B^{j+1}}b^{4}\left( \frac{\ell }{B^{j}}\right) 
\frac{\ell +\eta _{q}}{\eta _{q}\omega _{q}}\binom{\ell +q-2}{\ell },
\end{eqnarray*}%
where we used (see \cite{abramowitzstegun}, eq.22.4.2) 
\begin{equation}
\mathcal{C}_{\ell _{i}}^{\left( \eta _{q}\right) }\left( 1\right) =\binom{%
\ell +2\eta _{q}-1}{\ell }=\binom{\ell +q-2}{\ell }.
\label{gegenbauerin1}
\end{equation}%
Therefore,%
\begin{eqnarray*}
\left\Vert \widetilde{h}_{j,t}\star _{2}^{1}\widetilde{h}_{j,t}\right\Vert
_{L^{2}\left( \mu _{t}\right) }^{2} &=& R_{t}^{-2}\left( B^{\frac{q}{2}j}\sqrt{2\gamma _{j,q}}\right)
^{-4}B^{2qj}\left( \sum_{\ell }b^{4}\left( \frac{\ell }{B^{j}}\right) \frac{%
\ell +\eta _{q}}{\eta _{q}\omega _{q}}B^{-qj}\binom{\ell +q-2}{\ell }\right)
^{2}\int_{\mathbb{S}^{q}}\mu _{t}\left( dz\right)  \\
&=& O\left( R_{t}^{-1}\right).
\end{eqnarray*}
For the fifth and last condition, let $1\leq i_{1}\neq i_{2}\leq d$. We clearly have
\begin{equation*}
\left\langle \widetilde{h}_{j_{i_{1}},t},\widetilde{h}_{j_{i_{2}},t}\right%
\rangle _{L^{2}\left( \mu _{t}^{\otimes 2}\right) } = 0
\end{equation*}%
as by assumption we have $\left\vert
j_{i_{2}}-j_{i_{1}}\right\vert >1$ and hence, it is enough to exploit the
orthogonality properties of Gegenbauer polynomials. Gathering all these
estimates together yields 
\begin{equation*}
d_{2}\left( \widetilde{U_{j_{1},\ldots,j_{d}}^{\left( 2\right) }}\left(
t\right) ,Z_{d}\right) =O\left( \max_{j_{1},\ldots,j_{d}}\left( B^{\frac{q}{2}%
j}R_{t}^{-\frac{1}{2}}+B^{-\frac{q}{2}j}+R_{t}^{-\frac{1}{2}}\right) \right).
\end{equation*}%
The proof of Theorem \ref{maintheorem2} is hence concluded.

\section{Proof of Lemma \ref{l:dep}}\label{s:lemma} \noindent Using the classical theory of Hoeffding decompositions as in \cite[Section 3.6]{BouPec} and \cite{DyMa}, we infer that there exist nonnegative constants $u(n,l)\in [0,\infty)$, $1\leq l\leq m\leq n$ such that 
$$
{\rm Var} (U_n) = \sum_{l=1}^m \frac{n^l}{l!} u(n,l) \quad \text{and} \quad {\rm Var} (U'_n) = \sum_{l=1}^m \binom{n}{l} u(n,l),
$$
from which we deduce immediately that ${\rm Var}(U'_n)\to 1$ and also that each mapping $n \mapsto \binom{n}{l} u(n,l)$, $l=1,\ldots,m$, is necessarily bounded. We will now adopt the usual falling factorial notation, namely: $n_{[l]} = n(n-1)\cdots (n-l+1)$. Reasoning as in \cite[p. 785]{DyMa}, one infers that
$$
\mathbb{E}[(U_n-U'_n)^2] = \sum_{l=1}^m \binom{n}{l} u(n,l)\left\{ 1+\frac{n^l}{n_{(l)}}-2b(n,l)\right\},
$$
where $b(n,l) := e^{-n} \sum_{p=0}^\infty \frac{n^p}{p!} \binom{n \wedge p}{l} \binom{n}{l}^{-1}$. Since ${n^l}/{n_{[l]}}-1 = O(1/n)$, the conclusion is achieved once we show that $1 - b(n,l) = O(1/\sqrt{n})$, $l=1,...,m$. Elementary computations yield that
$$
1-b(n,l) = e^{-n}\sum_{p=n-l+1}^n \frac{n^p}{p!}  + \left(1-\frac{n^l}{n_{[l]}} \right) e^{-n} \sum_{p=0}^{n-l} \frac{n^p}{p!} =e^{-n}\sum_{p=n-l+1}^n \frac{n^p}{p!}  + O\left(\frac1n\right).
$$
By virtue of a standard application of Stirling's formula one has that, for $l=1,...,m$,
$$
e^{-n}\sum_{p=n-l+1}^n \frac{n^p}{p!} \sim n^{-1/2},
$$
and the desired conclusion follows at once.

\section{Auxiliary results}
\label{appendix}

\noindent Fix a Poisson measure $N_t$ with control $\mu_t$, $t>0$. Consider an integer $i\geq 1$ as well as a symmetric kernel $f\in L^2(\mu_t^i)$: we shall denote by $I_i(f)$ the usual Wiener-It\^o integral of order $i$, of $f$ with respect to $N_t$. See for instance \cite[Chapter 5]{PeTa} for a detailed discussion of this concept.

\subsection{Gaussian approximations using Stein-Malliavin methods}
\label{appendixsteinmall}

\noindent The following crucial fact is proved by Reitzner \& Schulte in \cite[Lemma
3.5 and Theorem 3.6]{lesmathias}:
\begin{proposition}
\label{p:zap} Consider a kernel $h\in L_{s}^{1}(\mu _{t}^{k})$ such that the
corresponding $U$-statistic $F$ in \eqref{e:ustat} (for the choice $N=N_t$) is square-integrable.
Then, $h$ is necessarily square-integrable, and $F$ admits a chaotic
decomposition of the form 
\begin{equation*}
F=\int_{Z^{k}}h\left( x_{1},\ldots ,x_{k}\right) d\mu _{t}^{k}+\sum_{i=1}^{\infty
}I_{i}\left( h_{i}\right) ,
\end{equation*}%
with 
\begin{equation*}
h_{i}(x_{1},\ldots ,x_{i})=\binom{k}{i}\int_{Z^{k-i}}h(x_{1},\ldots
,x_{i},x_{i+1},\ldots ,x_{k})d\mu _{t}^{k-i},\quad (x_{1},\ldots ,x_{i})\in
Z^{i},
\end{equation*}
for $1\leq i\leq k$, and $h_{i}=0$ for $i>k$. In particular, $h=h_{k}$ and
the projection $h_{i}$ is in $L_{s}^{1,2}(\mu ^{i})$ for each $1\leq i\leq k$.
\end{proposition}
\noindent We need also to recall two upper bounds involving random variables living in
the \textit{first Wiener chaos} associated to the Poisson measure $N$. The
first bound was proved in \cite{PSTU}, and concerns normal approximations in
dimension 1 with respect to the Wasserstein distance. The second bound
appears in \cite{PecZheng}, and provides estimates for multidimensional
normal approximations with respect to the distance $d_{2}$. Both bounds are
obtained by means of a combination of the Malliavin calculus of variations
and the Stein's method for probabilistic approximations. In what follows, we
shall use the symbols $N(f)$ and $\hat{N}(f)$, respectively, to denote the
Wiener-It\^{o} integrals of $f$ with respect to $N$ and with respect to the 
\textit{compensated Poisson measure} 
\begin{equation*}
\hat{N}(A)=N(A)-\mu (A),\quad A\in \mathcal{B}(\Theta ),
\end{equation*}
where we use the convention $N(A)-\mu (A)=\infty $ whenever $\mu (A)=\infty $
(recall that $\mu $ is $\sigma $-finite). We shall consider Wiener-It\^{o}
integrals of functions $f$ having the form $f=[0,t]\times h$, where $t>0$
and $h\in L^{2}(\mathbb{S}^{q},\nu )\cap L^{1}(\mathbb{S}^{q},\nu )$. For a
function $f$ of this type we simply write%
\begin{equation*}
N(f)=N([0,t]\times h):=N_{t}(h),\quad \mbox{and}\quad \hat{N}(f)=\hat{N}%
([0,t]\times h):=\hat{N}_{t}(h).
\end{equation*}

\begin{proposition}[Gaussian approximations in the linear regime (%
\protect\cite{PSTU}, \protect\cite{PecZheng})] Under the assumptions and notation of this section, let $%
h\in L^{2}(\mathbb{S}^{q},\nu ):=L^{2}(\nu )$, let $Z\sim \mathcal{N}(0,1)$
and fix $t>0$. Then, the following bound holds: 
\begin{equation*}
d_{W}(\hat{N}_{t}(h),Z)\leq \left\vert 1-\Vert h\Vert _{L^{2}(\mu
_{t})}^{2}\right\vert +\int_{\mathbb{S}^{q}}|h(z)|^{3}\mu _{t}(dz).
\end{equation*}
\end{proposition}
\noindent For a fixed integer $d\geq 1$, let $Z_{d}\sim \mathcal{N}_{d}\left( 0,\Sigma
\right) $ where $\Sigma $ is a positive definite covariance matrix and let 
\begin{equation*}
F_{t}=\left( F_{t,1},\ldots ,F_{t,d}\right) =\left( \hat{N}_{t}\left(
h_{t,1}\right) ,\ldots, \hat{N}_{t}\left( h_{t,d}\right) \right)
\end{equation*}%
be a collection of $d$-dimensional random vectors such that $h_{t,a}\in
L^{2}(\nu )$. If we call $\Gamma _{t}$ the covariance matrix of $F_{t}$,
then
\begin{eqnarray}
d_{2}\left( F_{t},Y\right) &\leq &\left\Vert \Sigma ^{-1}\right\Vert
_{op}\left\Vert \Sigma \right\Vert _{op}^{\frac{1}{2}}\left\Vert \Sigma
-\Gamma _{t}\right\Vert _{H.S.}  +\frac{\sqrt{2\pi }}{8}\left\Vert \Sigma ^{-1}\right\Vert _{op}^{\frac{3}{2%
}}\left\Vert \Sigma \right\Vert _{op}\sum_{i,j,k=1}^{d}\int_{\mathbb{S}%
^{q}}\left\vert h_{t,i}\left( z\right) \right\vert \left\vert h_{t,j}\left(
z\right) \right\vert \left\vert h_{t,k}\left( x\right) \right\vert \mu
_{t}\left( dx\right)\qquad  \label{e:yepes} \\
&\leq &\left\Vert \Sigma ^{-1}\right\Vert _{op}\left\Vert C\right\Vert
_{op}^{\frac{1}{2}}\left\Vert \Sigma -\Gamma _{t}\right\Vert _{H.S.} +\frac{d^{2}\sqrt{2\pi }}{8}\left\Vert \Sigma ^{-1}\right\Vert _{op}^{%
\frac{3}{2}}\left\Vert \Sigma \right\Vert _{op}\sum_{i=1}^{d}\int_{\mathbb{S}%
^{q}}\left\vert h_{t,i}\left( x\right) \right\vert ^{3}\mu _{t}\left(
dx\right)\notag ,
\end{eqnarray}%
where $\Vert \cdot \Vert _{op}$ and $\Vert \cdot \Vert _{H.S.}$ stand,
respectively, for the operator and Hilbert-Schmidt norms.
\begin{remark}
The estimate \eqref{e:yepes} is a direct consequence of Theorem 3.3 in \cite{PecZheng}.
\end{remark}
\begin{proposition}[Gaussian approximations in the quadratic regime (\cite{PSTU}, \cite{PecZheng})]
\label{Peccati} Let $d=d_{1}+d_{2}$, with $d_{1},d_{2}\geq 1$ two integers
and Let $Z_{d}\sim \mathcal{N}\left( 0,I_{d}\right) $. Assume that 
\begin{equation*}
F_{j}=\left( F_{j,1},\ldots,F_{j,d}\right) :=\left( I_{1}\left( g_{j,1}\right)
,\ldots,I_{1}\left( g_{j,d_{1}}\right) ,I_{2}\left( h_{j,1}\right)
,\ldots,I_{2}\left( h_{j,d_{2}}\right) \right) ,
\end{equation*}%
where the symmetric kernels $g_{j,1},\ldots,g_{j,d_{1}},h_{j,1},\ldots,h_{j,d_{2}}$ satisfy
the following conditions: for $k=1,\ldots,d_{1}$, $g_{j,k}\in L^{2}\left( \mu
_{t}\right) \cap L^{3}\left( \mu _{t}\right) $; for $k=1,\ldots,d_{2}$,$%
h_{j,k}\in L^{2}\left( \mu _{t}^{\otimes 2}\right) $ is such that (a) for $%
1\leq k_{1},k_{2}\leq d_{2}$, $h_{j,k_{1}}\star _{2}^{1}h_{j,k_{2}}\in
L^{2}\left( \mu _{t}\right) $, (b) $h_{j,k}\in L^{4}\left( \mu _{t}^{\otimes
2}\right) $, (c) $\left\vert h_{j,k_{1}}\right\vert \star _{2}^{1}\left\vert
h_{j,k_{2}}\right\vert $, $\left\vert h_{j,k_{1}}\right\vert \star
_{2}^{0}\left\vert h_{j,k_{2}}\right\vert $ and $\left\vert
h_{j,k_{1}}\right\vert \star _{1}^{0}\left\vert h_{j,k_{2}}\right\vert $ are
well defined and finite for every value of their arguments and (d) it holds
that, for $1\leq k_{1},k_{2}\leq d_{2}$, 
\begin{equation*}
\int_{\mathbb{S}^{q}}\sqrt{\int_{\mathbb{S}^{q}}h_{j,k_{1}}^{2}\left(
z_{1},z_{2}\right) h_{j,k_{2}}^{2}\left( z_{1},z_{2}\right) \mu _{t}\left(
dz_{2}\right) }\mu _{t}\left( dz_{1}\right) <\infty .
\end{equation*}%
Assume that the following five conditions hold:
\begin{enumerate}
\item for $k=1,\ldots,d_{1}$, $\left\Vert g_{j,k}\right\Vert _{L^{4}(\mu _{t})}\rightarrow 0$;
\item for $k=1,\ldots,d_{1}$, $\left\Vert h_{j,k}\right\Vert _{L^{4}(\mu _{t}^{\otimes 2})}\rightarrow 0$;
\item for $1\leq k\leq d_{2}$, $\left\Vert h_{j,k}\star _{1}^{1}h_{j,k}\right\Vert _{L^{2}(\mu _{t}^{\otimes
2})}\rightarrow 0$, where $
\left( h_{j,k}\star _{1}^{1}h_{j,k}\right) \left( z_{1},z_{2}\right) =\int_{%
\mathbb{S}^{q}}h_{j,k}\left( z_{1},z_{3}\right) h_{j,k}\left(
z_{2},z_{3}\right) \mu _{t}\left( dz\right)$;
\item for $1\leq k\leq d_{2}$, $\left\Vert h_{j,k}\star _{2}^{1}h_{j,k}\right\Vert _{L^{2}(\mu
_{t})}\rightarrow 0$, where $(h_{j,k}\star _{2}^{1}h_{j,k})\left( z\right) =\int_{\mathbb{S}%
^{q}}h_{j,k}^{2}\left( z,a\right) \mu _{t}\left( da\right)$;
\item for $1\leq k_{1},k_{2}\leq d$, $\mathbb{E}\left[ F_{j,k_{1}}F_{j,k_{2}}\right]
 \rightarrow \delta _{k_{1}}^{k_{2}}$.
\end{enumerate}
Then $F_{j}$ converges in distribution to $Z_{d}$ and
\begin{eqnarray*}
d_{2}\left( F_{j},Z_{d}\right) & \leq & \frac{1}{2}\sqrt{\sum_{k_{1},k_{2}=1}^{d_{2}}\left( 4\left\Vert
h_{j,k_{1}}\star _{2}^{1}h_{j,k_{2}}\right\Vert _{L^{2}(\mu
_{t})}^{2}+8\left\Vert h_{j,k_{1}}\star _{1}^{1}h_{j,k_{1}}\right\Vert
_{L^{2}(\mu _{t}^{\otimes 2})}^{2}\right)
+5\sum_{k_{1}=1}^{d_{1}}\sum_{k_{2}=1}^{d_{2}}5\left\Vert g_{j,k_{1}}\star
_{1}^{1}h_{j,k_{2}}\right\Vert _{L^{2}(\mu _{t})}^{2}} \\
&& +d_{1}^{2}\sum_{k=1}^{d_{1}}\left\Vert g_{j,k}\right\Vert _{L^{3}(\mu
_{t})}^{3}+8d_{2}^{2}\sum_{k=1}^{d_{2}}\left\Vert h_{j,k}\right\Vert
_{L^{2}(\mu _{t}^{\otimes 2})}\left( \left\Vert h_{j,k}\right\Vert
_{L^{4}(\mu _{t}^{\otimes 2})}+\sqrt{2}\left\Vert h_{j,k}\star
_{2}^{1}h_{j,k}\right\Vert _{L^{2}(\mu _{t})}^{2}\right) .
\end{eqnarray*}
\end{proposition}
\subsection{Functional inequalities for needlets kernels}
\noindent We present here some functional inequalities which are necessary for our
main arguments. The first Lemma is basically a consequence of \eqref{Lpbound}. For $z_{i}\in \mathbb{S}^{q}$, $i=1,\ldots,D$, let
\begin{equation*}
\Omega =\left\{ v_{1},\ldots,v_{D}\colon \sum_{i=1}^{D}v_{i}=V,v_{i}\neq v_{j}\ 
\forall i\neq j\right\} \quad \text{and} \quad L_{V}\left( z_{1},z_{2},\ldots,z_{D}\right) =\sum_{\left\{
v_{1},\ldots,v_{D}\right\} \in \Omega }c_{v_{1},\ldots,v_{D}}\prod_{i=1}^{D}\psi
_{jk}\left( z_{i}\right) ^{v_{i}}.
\end{equation*}%
\begin{lemma}
For $C_{v_{i}}$ as defined in \eqref{Lpbound} and denoting by $\delta _{0}^{k}$ the Kronecker delta function, it holds that 
\begin{equation*}
\left\vert \int_{\left( \mathbb{S}^{q}\right) \ ^{D}}L_{V}\left(
z_{1},z_{2},\ldots,z_{D}\right) \mu ^{\otimes D}\left(
dz_{1},dz_{2},\ldots,dz_{D}\right) \right\vert \leq \sum_{\left\{ v_{1},\ldots,v_{D}\right\} \in \Omega
}c_{v_{1},\ldots,v_{D}}\left( \prod_{i=1}^{D}C_{v_{i}}\right)B^{jq\left( \frac{1}{2}V-D+\sum_{i=1}^{D}\delta _{0}^{k_{i}}\right) }.
\end{equation*}
\end{lemma}
\begin{proof}
Easy calculations lead to 
\begin{eqnarray*}
&& \left\vert \int_{\left( \mathbb{S}^{q}\right) \ ^{D}}L_{V}\left(
z_{1},z_{2},\ldots,z_{D}\right) \mu ^{\otimes D}\left(
dz_{1},dz_{2},\ldots,dz_{D}\right) \right\vert \leq \int_{\left( \mathbb{S}^{q}\right) \ ^{D}}\left\vert L_{V}\left(
z_{1},z_{2},\ldots,z_{D}\right) \right\vert \mu ^{\otimes D}\left(
dz_{1},dz_{2},\ldots,dz_{D}\right) \\
&& \leq \sum_{\left\{ v_{1},\ldots,v_{D}\right\} \in \Omega
}c_{v_{1},\ldots,v_{D}}\prod_{i=1}^{D}\int_{\mathbb{S}^{q}}\left\vert \psi
_{jk}\left( z_{i}\right) \right\vert ^{v_{i}}\mu \left( dz_{i}\right) \leq
M\sum_{\left\{ v_{1},\ldots,v_{D}\right\} \in \Omega
}c_{v_{1},\ldots,v_{D}}\prod_{i=1}^{D}\left\Vert \psi _{j}\right\Vert
_{L^{k_{i}}\left( dz\right) }^{k_{i}}.
\end{eqnarray*}%
For any $i$, it follows from \eqref{Lpbound} that 
\begin{equation*}
\left\Vert \psi _{jk} \right\Vert _{L^{v_{i}}\left( dz \right)  }^{v_{i}}\leq C_{v_{i}} B^{jq\left( \frac{v_{i}}{2}-1\right) }\text{  if }%
v_{i}\neq 0\text{ }{1\text{ \ \ \ \ \ \ \ \ if }v_{i}=0}=C_{v_{i}}B^{j\left[
q\left( \frac{k_{i}}{2}-1\right) +q\delta _{k_{i}}^{0}\right] }.
\end{equation*}%
Therefore we obtain 
\begin{eqnarray*}
\left\vert \int_{\left( \mathbb{S}^{q}\right) \ ^{D}}L_{V}\left(
z_{1},z_{2},\ldots,z_{D}\right) \mu ^{\otimes D}\left(
dz_{1},dz_{2},\ldots,dz_{D}\right) \right\vert &\leq &M\sum_{\left\{ v_{1},\ldots,v_{D}\right\} \in \Omega
}c_{v_{1},\ldots,v_{D}}\left( \prod_{i=1}^{D}C_{v_{i}}\right)
B^{qj\sum_{i=1}^{D}\left[ \left( \frac{v_{i}}{2}-1\right) +\delta
_{k_{i}}^{0}\right] } \\
&=&M\sum_{\left\{ v_{1},\ldots,v_{D}\right\} \in \Omega
}c_{v_{1},\ldots,v_{D}}\left( \prod_{i=1}^{D}C_{v_{i}}\right) B^{qj\left( \frac{%
1}{2}V-D+\sum_{i=1}^{D}\delta _{k_{i}}^{0}\right) }.
\end{eqnarray*}
\end{proof}
\noindent Now let $\widetilde{\Omega }$ $\subset \Omega $, $\widetilde{\Omega }\neq
\varnothing $, and for any set $\left\{ v_{1},\ldots,v_{D}\right\} \in $ $%
\widetilde{\Omega }$ write $Z_{\left\{ v_{1},\ldots,v_{D}\right\} }:=\sharp \left\{ i\colon v_{i}=0,\ \nu
_{i}\in \left\{ v_{1},\ldots,v_{D}\right\} \right\}$ and $Z_{0}=\max_{\left\{ v_{1},\ldots,v_{D}\right\} \in \widetilde{\Omega }%
}Z_{\left\{ v_{1},\ldots,v_{D}\right\} }$.
\begin{corollary}
\label{corbound}There exists $C^{^{\prime }}>0$ such that%
\begin{equation*}
\left\vert \int_{\left( \mathbb{S}^{q}\right) \ ^{D}}L_{V}\left(
z_{1},z_{2},\ldots,z_{D}\right) \mu ^{\otimes D}\left(
dz_{1},dz_{2},\ldots,dz_{D}\right) \right\vert \leq C^{^{\prime }}B^{jq\left( 
\frac{1}{2}V-D+N_{0}\right) }.
\end{equation*}
\end{corollary}
\begin{proof}
We have%
\begin{eqnarray*}
\left\vert \int_{\left( \mathbb{S}^{q}\right) \ ^{D}}L_{V}\left(
z_{1},z_{2},\ldots,z_{D}\right) \mu ^{\otimes D}\left(
dz_{1},dz_{2},\ldots,dz_{D}\right) \right\vert & \leq & \sum_{\left\{ v_{1},\ldots,v_{D}\right\} \in \Omega
}c_{v_{1},\ldots,v_{D}}\left( \prod_{i=1}^{D}C_{v_{i}}\right) B^{qj\left( \frac{%
1}{2}V-D+\sum_{i=1}^{D}\delta _{k_{i}}^{0}\right) } \\
&\leq & C^{^{\prime }}MB^{qj\left( \frac{1}{2}V-D+\sum_{i=1}^{D}\delta
_{k_{i}}^{0}\right) }\leq C^{^{\prime }}MB^{jq\left( \frac{1}{2}%
V-D+N_{0}\right) }.
\end{eqnarray*}
\end{proof}
\begin{lemma}
\label{lemma_int}For any $j$, $k_{1}\neq k_{2}=1,\ldots,K_{j}$, $\tau \geq 2,$ $%
n_{1},n_{2}>1$, we have%
\begin{equation*}
\int_{\mathbb{S}^{q}}\psi _{jk_{1}}^{n_{1}}\left( z\right) \psi
_{jk_{2}}^{n_{2}}\left( z\right) \mu \left( dz\right) \leq C_{\tau ,M,n_{1},n_{2}}B^{\left( \frac{\left( n_{1}+n_{2}\right) }{2}%
-1\right) qj}\left( \frac{1}{\left( 1+B^{\frac{q}{2}j}d\left( \xi
_{jk_{1}},\xi _{jk_{2}}\right) \right) ^{\min \left( n_{1},n_{2}\right) \tau
}}\right).
\end{equation*}
\end{lemma}
\begin{proof}
As in \cite{dmp2014}, we split the sphere into two regions
\begin{equation*}
S_{1} =\left\{ z\in \mathbb{S}^{q}\colon d\left( z,\xi _{jk_{1}}\right) >d\left(
\xi _{jk_{1}},\xi _{jk_{2}}\right) /2\right\} \quad \text{and} \quad S_{2} =\left\{ z\in \mathbb{S}^{q}\colon d\left( z,\xi _{jk_{2}}\right) >d\left(
\xi _{jk_{1}},\xi _{jk_{2}}\right) /2\right\},
\end{equation*}%
so that $\mathbb{S}^{q}\subseteq S_{1}\cup S_{2}$. On the other hand, we
have by \eqref{localization} that there exists $\tau >2$ such that
\begin{eqnarray*}
\int_{\mathbb{S}^{q}}\psi _{jk_{1}}^{n_{1}}\left( z\right) \psi
_{jk_{2}}^{n_{2}}\left( z\right) \mu \left( dz\right) &\leq & \kappa _{\tau }^{n_{1}+n_{2}}M\int_{\mathbb{S}^{q}}\frac{B^{n_{1}%
\frac{q}{2}j}}{\left( 1+B^{\frac{q}{2}j}d\left( z,\xi _{jk_{1}}\right)
\right) ^{n_{1}\tau }}\frac{B^{n_{2}\frac{q}{2}j}}{\left( 1+B^{\frac{q}{2}%
j}d\left( z,\xi _{jk_{2}}\right) \right) ^{n_{2}\tau }}dz \\
&\leq &\kappa _{\tau }^{n_{1}+n_{2}}M\left[ \int_{S_{1}}\frac{B^{n_{1}\frac{q%
}{2}j}}{\left( 1+B^{\frac{q}{2}j}d\left( z,\xi _{jk_{1}}\right) \right)
^{n_{1}\tau }}\frac{B^{n_{2}\frac{q}{2}j}}{\left( 1+B^{\frac{q}{2}j}d\left(
z,\xi _{jk_{2}}\right) \right) ^{n_{2}\tau }}dz\right. \\
&&+\left. \int_{S_{2}}\frac{B^{n_{1}\frac{q}{2}j}}{\left( 1+B^{\frac{q}{2}%
j}d\left( z,\xi _{jk_{1}}\right) \right) ^{n_{1}\tau }}\frac{B^{n_{2}\frac{q%
}{2}j}}{\left( 1+B^{\frac{q}{2}j}d\left( z,\xi _{jk_{2}}\right) \right)
^{n_{2}\tau }}dz\right].
\end{eqnarray*}%
Now, observe that 
\begin{eqnarray*}
&& B^{\frac{\left( n_{1}+n_{2}\right) }{2}qj}\int_{S_{1}}\frac{1}{\left( 1+B^{%
\frac{q}{2}j}d\left( z,\xi _{jk_{1}}\right) \right) ^{n_{1}\tau }}\frac{1}{%
\left( 1+B^{\frac{q}{2}j}d\left( z,\xi _{jk_{2}}\right) \right) ^{n_{2}\tau }%
}dz \\
&& \leq \frac{B^{\frac{\left( n_{1}+n_{2}\right) }{2}qj}}{\left( 1+B^{\frac{q}{2%
}j}d\left( \xi _{jk_{1}},\xi _{jk_{2}}\right) \right) ^{n_{2}\tau }}%
\int_{S_{1}}\frac{1}{\left( 1+B^{\frac{q}{2}j}d\left( z,\xi _{jk_{1}}\right)
\right) ^{n_{1}\tau }}dz \\
&& \leq \frac{2\pi B^{\frac{\left( n_{1}+n_{2}\right) }{2}qj}}{\left( 1+B^{%
\frac{q}{2}j}d\left( \xi _{jk_{1}},\xi _{jk_{2}}\right) \right) ^{n_{2}\tau }%
}\int_{0}^{\pi }\frac{\sin \theta }{\left( 1+B^{\frac{q}{2}j}\theta \right)
^{n_{1}\tau }}d\theta \leq \frac{2\pi B^{\frac{\left( n_{1}+n_{2}\right) }{2}%
qj}B^{-qj}}{\left( 1+B^{\frac{q}{2}j}d\left( \xi _{jk_{1}},\xi
_{jk_{2}}\right) \right) ^{n_{2}\tau }}\left( \int_{0}^{\infty }\frac{ydy}{%
1+y^{n_{1}\tau }}\right) \\
&& \leq \frac{2\pi B^{\frac{\left( n_{1}+n_{2}\right) }{2}qj}B^{-qj}}{\left(
1+B^{\frac{q}{2}j}d\left( \xi _{jk_{1}},\xi _{jk_{2}}\right) \right)
^{n_{2}\tau }}\left( \int_{0}^{1}\frac{ydy}{1+y^{n_{1}\tau }}%
+\int_{1}^{\infty }\frac{ydy}{1+y^{n_{1}\tau }}\right) \leq \frac{2\pi CB^{%
\frac{\left( n_{1}+n_{2}\right) }{2}qj}B^{-qj}}{\left( 1+B^{\frac{q}{2}%
j}d\left( \xi _{jk_{1}},\xi _{jk_{2}}\right) \right) ^{n_{2}\tau }}.
\end{eqnarray*}
The same result is obtained for $S_{2}$, so that 
\begin{eqnarray*}
\int_{\mathbb{S}^{q}}\psi _{jk_{1}}^{n_{1}}\left( z\right) \psi
_{jk_{2}}^{n_{2}}\left( z\right) \mu \left( dz\right) &\leq & C_{\tau ,M,n_{1},n_{2}}B^{\left( \frac{\left( n_{1}+n_{2}\right) }{2}%
-1\right) dq}\left( \frac{1}{\left( 1+B^{\frac{q}{2}j}d\left( \xi
_{jk_{1}},\xi _{jk_{2}}\right) \right) ^{n_{2}\tau }}+\frac{1}{\left( 1+B^{%
\frac{q}{2}j}d\left( \xi _{jk_{1}},\xi _{jk_{2}}\right) \right) ^{n_{2}\tau }%
}\right) \\
&\leq & C_{\tau ,M,n_{1},n_{2}}B^{\left( \frac{\left( n_{1}+n_{2}\right) }{2}%
-1\right) qj}\left( \frac{1}{\left( 1+B^{\frac{q}{2}j}d\left( \xi
_{jk_{1}},\xi _{jk_{2}}\right) \right) ^{\min \left( n_{1},n_{2}\right) \tau
}}\right),
\end{eqnarray*}%
as claimed.
\end{proof}
\begin{lemma}
\label{lemma3needlets}It holds that%
\begin{equation*}
\sum_{k_{1},k_{2},k_{3}=1}^{d}\int_{\mathbb{S}^{q}}\left\vert \psi
_{jk_{1}}^{u}\left( z\right) \right\vert \left\vert \psi _{jk_{2}}^{u}\left(
z\right) \right\vert \left\vert \psi _{jk_{3}}^{u}\left( z\right)
\right\vert dz\leq dC_{u}\kappa _{t}^{\prime \prime \prime }B^{\frac{3}{2}%
quj}.
\end{equation*}
\end{lemma}
\begin{proof}
Following similar arguments to those in \cite{dmp2014}, we have%
\begin{equation*}
\sum_{k_{1},k_{2},k_{3}=1}^{d}\int_{\mathbb{S}^{q}}\left\vert \psi
_{jk_{1}}^{u}\left( z\right) \right\vert \left\vert \psi _{jk_{2}}^{u}\left(
z\right) \right\vert \left\vert \psi _{jk_{3}}^{u}\left( z\right)
\right\vert dz \leq C\sum_{\lambda }\int_{\mathcal{B}\left( \xi _{j\lambda },B^{-qj}\right)
}\left( \sum_{k=1}^{d}\left\vert \psi _{jk}^{u}\left( z\right) \right\vert
\right) ^{3}dz,
\end{equation*}%
so that there exists $\tau >2$ such that
\begin{equation*}
\sum_{k=1}^{d}\left\vert \psi _{jk}^{u}\left( z\right) \right\vert \leq
\sum_{k=1}^{d}\frac{\kappa _{\tau }^{u}B^{\frac{q}{2}uj}}{\left( 1+B^{\frac{%
d}{2}j}d\left( \xi _{jk},z\right) \right) ^{u\tau }} \leq \kappa _{\tau }^{u}B^{\frac{q}{2}uj}+\sum_{k\colon \xi _{jk}\notin \mathcal{%
B}\left( \xi _{j\lambda },B^{-qj}\right) }^{d}\frac{\kappa _{\tau }^{u}B^{%
\frac{q}{2}uj}}{\left( B^{\frac{q}{2}j}d\left( \xi _{jk},\xi _{j\lambda
}\right) \right) ^{u\tau }}.
\end{equation*}%
For $\xi _{jk}\notin \mathcal{B}\left( \xi _{j\lambda },B^{-qj}\right) $, $%
z\notin \mathcal{B}\left( \xi _{j\lambda },B^{-qj}\right) $, using the triangle inequality yields $d\left( \xi _{jk},\xi _{j\lambda }\right) +d\left( \xi _{jk},z\right) \geq
d\left( z,\xi _{j\lambda }\right)$. Using the fact that $d\left( \xi _{jk},\xi _{j\lambda }\right) \geq d\left( \xi _{jk},z\right)$ and $2d\left( \xi _{jk},\xi _{j\lambda }\right) \geq d\left( z,\xi
_{j\lambda }\right)$, we obtain
\begin{eqnarray*}
\sum_{k\colon \xi _{jk}\notin \mathcal{B}\left( \xi _{j\lambda },B^{-qj}\right)
}^{d}\frac{\kappa _{\tau }^{u}B^{\frac{q}{2}uj}}{\left( B^{\frac{q}{2}%
j}d\left( \xi _{jk},\xi _{j\lambda }\right) \right) ^{u\tau }}
&\leq & \sum_{k\colon \xi _{jk}\notin \mathcal{B}\left( \xi _{j\lambda
},B^{-qj}\right) }^{d}\frac{1}{meas\left( \mathcal{B}\left( \xi _{j\lambda
},B^{-qj}\right) \right) }\int_{\mathcal{B}\left( \xi _{j\lambda
},B^{-qj}\right) }\frac{\kappa _{\tau }^{u}B^{\frac{q}{2}uj}}{\left( B^{%
\frac{q}{2}j}d\left( \xi _{jk},\xi _{j\lambda }\right) \right) ^{u\tau }}dz \\
&\leq & \kappa _{\tau }^{\prime }B^{\frac{q}{2}uj},
\end{eqnarray*}%
where we applied \cite[Lemma 6]{bkmpAoS}. We obtain $\sum_{k}^{d}\left\vert \psi _{jk}^{u}\left( z\right) \right\vert \leq \kappa
_{t}^{\prime \prime }B^{\frac{q}{2}uj}$ uniformly over $z\in \mathbb{S}^{q}$. Finally, in order to have
\begin{equation*}
\int_{\mathbb{S}^{q}}\left( \sum_{k}^{d}\left\vert \psi _{jk}^{u}\left(
z\right) \right\vert \right) ^{3}dz\leq dC_{u}\kappa _{t}^{\prime \prime
\prime }B^{\frac{3}{2}uqj},
\end{equation*}%
it is enough to check that
\begin{equation*}
\sum_{k_{1},k_{2},k_{3}=1}^{d}\int_{\mathbb{S}^{q}}\left( \left\vert \psi
_{jk_{1}}^{u}\left( z\right) \right\vert \left\vert \psi _{jk_{2}}^{u}\left(
z\right) \right\vert \left\vert \psi _{jk_{3}}^{u}\left( z\right)
\right\vert \right) dx \leq C_{u}\kappa _{t}^{\prime \prime \prime }B^{quj}\sum_{k=1}^{d}\int_{%
\mathbb{S}^{q}}\left\vert \psi _{jk}^{u}\left( z\right) \right\vert dz\leq
dC_{u}\kappa _{t}^{\prime \prime \prime }B^{\frac{3}{2}quj}.
\end{equation*}
\end{proof}
\begin{lemma}
\label{rewritekernel} For $z_{1},z_{2}\in \mathbb{S}^{q}$, it holds that 
\begin{equation*}
\sum_{k}\psi _{jk}\left( z_{1}\right) \psi _{jk}\left( z_{2}\right)
=\sum_{\ell }b^{2}\left( \frac{\ell }{B^{j}}\right) \frac{\ell +\eta _{q}}{%
\eta _{q}\omega _{q}}\mathcal{C}_{\ell }^{\left( \eta _{q}\right) }\left(
\left\langle z_{1},z_{2}\right\rangle \right).
\end{equation*}
\end{lemma}
\begin{proof}
First observe that
\begin{equation*}
\sum_{k}\psi _{jk}\left( z_{1}\right) \psi _{jk}\left( z_{2}\right)=\sum_{k}\sum_{\ell _{1},\ell _{2}}\left( \prod_{i=1,2}b\left( \frac{\ell
_{i}}{B^{j}}\right) \frac{\ell _{i}+\eta _{q}}{\eta _{q}\omega _{q}}\mathcal{%
C}_{\ell _{i}}^{\left( \eta _{q}\right) }\left( \left\langle z_{i},\xi
_{jk}\right\rangle \right) \right) \lambda _{jk}.
\end{equation*}%
Using the cubature formula over the sphere (see \cite{npw1}) along with the
self--reproducing property of the Gegenbauer polynomials (see for instance 
\cite{szego}), we have
\begin{equation*}
\sum_{k}^{\left( \eta _{q}\right) }\mathcal{C}_{\ell _{1}}^{\left( \eta
_{q}\right) }\left( \left\langle z_{1},\xi _{jk}\right\rangle \right) 
\mathcal{C}_{\ell _{2}}^{\left( \eta _{q}\right) }\left( \left\langle
z_{2},\xi _{jk}\right\rangle \right) \lambda _{jk} = \int_{\mathbb{S}^{q}}%
\mathcal{C}_{\ell _{1}}^{\left( \eta _{q}\right) }\left( \left\langle
z_{1},\xi _{jk}\right\rangle \right) \mathcal{C}_{\ell _{2}}^{\left( \eta
_{q}\right) }\left( \left\langle z_{2},\xi _{jk}\right\rangle \right) dx =\left( \frac{\ell _{1}+\eta _{q}}{\eta _{q}\omega _{q}}\right) ^{-1}%
\mathcal{C}_{\ell _{1}}^{\left( \eta _{q}\right) }\left( \left\langle
z_{1},z_{2}\right\rangle \right) \delta _{\ell _{1}}^{\ell _{2}},
\end{equation*}
where $\delta _{y}^{x}$ is the Kronecker delta function. The statement follows
immediately.
\end{proof}
\begin{lemma}
\label{lemmawithsinteger}For any $s\in \mathbb{N}$%
\begin{equation}
\int_{\left( \mathbb{S}^{q}\right) ^{2}}\left( \sum_{\ell }b^{s}\left( \frac{%
\ell }{B^{j}}\right) \frac{\ell +\eta _{q}}{\eta _{q}\omega _{q}}\mathcal{C}%
_{\ell }^{\left( \eta _{q}\right) }\left( \left\langle
z_{1},z_{2}\right\rangle \right) \right) ^{n}dz_{2}dz_{1}=O\left(
B^{jq(n-2)}\right).  \label{rewriteas}
\end{equation}
\end{lemma}
\begin{proof}
For any~$s\in \mathbb{N}$, observe that the integrand in \eqref{rewriteas},
up to a factor $B^{jn}$, behaves as the $n$-th power of $\psi _{j}^{\left(
s\right) }\left( z\right) $ defined in \eqref{needletbehave}, as stated in
Remark \ref{stilltobewritten}. Hence we have 
\begin{equation*}
\int_{\left( \mathbb{S}^{q}\right) ^{2}}\left( \sum_{\ell }b^{2}\left( \frac{%
\ell }{B^{j}}\right) \frac{\ell +\eta _{q}}{\eta _{q}\omega _{q}}\mathcal{C}%
_{\ell }^{\left( \eta _{q}\right) }\left( \left\langle
z_{1},z_{2}\right\rangle \right) \right) ^{n}dz_{2}dz_{1} =O\left(
\left\Vert \psi _{j}^{(s)}\right\Vert _{L^{n}(dz)}^{n}\right) = O\left( B^{jq(n-2)}\right).
\end{equation*}
\end{proof}
\subsection{Auxiliary results related to the proof of Theorem \protect\ref%
{maintheorem1}}
\begin{lemma}
\label{lemmagtogamma}For any $j,k$, let $U_{jk}^{\left( 1\right) }\left(
t\right) $ be given by \eqref{general1bis} and, let $\Gamma _{1}\left(
j\right) ,$ $\Gamma _{21}\left( j\right) $ and $\Gamma _{22}\left( j\right) $
be given respectively by \eqref{gamma1}, \eqref{gamma2} and \eqref{gamma3}.
It holds that 
\begin{equation*}
\mathbb{E}\left[ U_{jk}^{\left( 1\right) }\left( t\right) \right]
=R_{t}^{2}\Gamma _{1}(j) \quad \text{and} \quad {\rm Var}\left( U_{jk}^{\left( 1\right) }\left( t\right) \right)
=R_{t}^{3}\Gamma _{2}(j)+R_{t}^{2}\Gamma _{3}(j)
\end{equation*}
\end{lemma}
\begin{proof}
We can easily observe that%
\begin{equation*}
\mathbb{E}\left[ U_{jk}^{\left( 1\right) }\left( t\right) \right]
=\int_{\left( \mathbb{S}^{q}\right) ^{2}}\sum_{i=0}^{u}\binom{u}{i}\psi
_{jk}^{u-r}\left( x\right) \psi _{jk}^{r}\left( y\right) \mu _{t}^{\otimes
2}\left( dx,dy\right) = R_{t}^{2}\sum_{r=0}^{u}\binom{u}{r}G_{u-r}\left( j\right) G_{r}\left(
j\right) =R_{t}^{2}\Gamma _{1}\left( j\right).
\end{equation*}
On the other hand,
\begin{equation*}
{\rm Var}\left( U_{j}^{\left( 1\right) }\left( t\right) \right) =\left\Vert
g_{jk}\right\Vert _{L^{2}\left( \mu _{t}\right) }^{2}+\left\Vert
h_{jk}\right\Vert _{L^{2}\left( \mu _{t}^{\otimes 2}\right) }^{2},
\end{equation*}%
where 
\begin{equation*}
L^{2}\left( \mu _{t}^{\otimes 2}\right) \ni h_{jk}\colon \left(
z_{1,}z_{2}\right) \mapsto \sum_{i=0}^{u}\binom{u}{i}\psi _{jk}^{u-r}\left(
z_{1}\right) \psi _{jk}^{r}\left( z_{2}\right)\quad \text{and} \quad L^{2}\left( \mu _{t}\right) \ni g_{jk}\colon z\mapsto \int_{\mathbb{S}%
^{q}}h_{jk}\left( z,a\right) \mu _{t}\left( da\right).
\end{equation*}
We hence have
\begin{eqnarray*}
\left\Vert g_{jk}\right\Vert _{L^{2}\left( \mu _{t}\right) }^{2}
&=&R_{t}^{3}\int_{\mathbb{S}^{q}}\left( \sum_{r=0}^{u}\binom{u}{r}\psi
_{jk}^{u-r}\left( z\right) G_{r}\left( j\right) \right) ^{2}\mu \left(
dz\right) =R_{t}^{3}\sum_{s,r=0}^{u}\binom{u}{s}\binom{u}{r}G_{s}\left( j\right)
G_{r}\left( j\right) \int_{\mathbb{S}^{q}}\psi _{jk}^{2u-\left( s+r\right)
}\left( z\right) \mu \left( dz\right) \\
&=&R_{t}^{3}\sum_{s,r=0}^{u}\binom{u}{s}\binom{u}{r}G_{s}\left( j\right)
G_{r}\left( j\right) G_{2u-\left( s+r\right) }(j)=R_{t}^{3}\Gamma _{2}\left(
j\right)
\end{eqnarray*}%
and%
\begin{eqnarray*}
\left\Vert h_{jk}\right\Vert _{L^{2}\left( \mu _{t}^{\otimes 2}\right) }^{2}
&=&\int_{\left( \mathbb{S}^{q}\right) ^{\otimes 2}}\left( \sum_{r=0}^{u}%
\binom{u}{i}\psi _{jk}^{u-r}\left( z_{1}\right) \psi _{jk}^{r}\left(
z_{2}\right) \right) ^{2}\mu _{t}^{\otimes 2}\left( dz_{1},dz_{2}\right) \\
&=& \int_{\left( \mathbb{S}^{q}\right) ^{\otimes 2}}\left( \sum_{r=0}^{u}%
\binom{u}{r}\psi _{jk}^{u-r}\left( z_{1}\right) \psi _{jk}^{r}\left(
z_{2}\right) \right) ^{2}\mu _{t}^{\otimes 2}\left( dz_{1},dz_{2}\right) = R_{t}^{2}\sum_{s,r=0}^{u}\binom{u}{s}\binom{u}{r}G_{s+r}\left( j\right)
G_{2u-\left( s+r\right) }\left( j\right)\\
&=& R_{t}^{2}\Gamma _{3}\left(
j\right) .
\end{eqnarray*}
\end{proof}
\begin{lemma}
\label{Lemma_gamma}Let $\Gamma _{1}\left( j\right) $, $\Gamma _{21}\left(
j\right) $ and $\Gamma _{22}\left( j\right) $ be given respectively by \eqref{gamma1}, \eqref{gamma2} and \eqref{gamma3}. Then, there exist $%
c_{u,m},C_{u,m},c_{2u,m}^{\left( 1\right) },C_{2u,M}^{\left( 1\right)
},c_{2u,m}^{\left( 2\right) },C_{2u,M}^{\left( 2\right) }>0$ such that 
\begin{equation*}
c_{u,m}B^{jq\left( \frac{u}{2}-1\right) }\left( 1+o_{j}\left( 1\right)
\right) \leq \left\vert \Gamma _{1}\left( j\right) \right\vert \leq
C_{u,M}B^{jq\left( \frac{u}{2}-1\right) }\left( 1+o_{j}\left( 1\right)
\right) , 
\end{equation*}
\begin{equation*}
c_{2u,m}^{\left( 1\right) }B^{j\frac{q}{2}\left( u-1\right) }\left(
1+o_{j}\left( 1\right) \right) \leq \left\vert \Gamma _{21}\left( j\right)
\right\vert \leq C_{2u,M}^{\left( 1\right) }B^{j\frac{q}{2}\left( u-1\right)
}\left( 1+o_{j}\left( 1\right) \right) , 
\end{equation*}
\begin{equation*}
c_{2u,m}^{\left( 2\right) }B^{j\frac{q}{2}\left( u-1\right) }\left(
1+o_{j}\left( 1\right) \right) \leq \left\vert \Gamma _{22}\left( j\right)
\right\vert \leq C_{2u,M}^{\left( 2\right) }B^{j\frac{q}{2}\left( u-1\right)
}\left( 1+o_{j}\left( 1\right) \right) . 
\end{equation*}
\end{lemma}

\begin{proof}
From \eqref{G_bound}, we have that 
\begin{equation*}
\left\vert \Gamma _{1}\left( j\right) \right\vert =\left\vert \sum_{i=0}^{u}%
\binom{u}{i}G_{u-i}\left( j\right) G_{i}\left( j\right) \right\vert \leq
\sum_{i=0}^{u}\binom{u}{i}\left\vert G_{u-i}(j)G_{i}(j)\right\vert.
\end{equation*}%
Using Corollary \ref{corbound}, we obtain
\begin{equation*}
\left\vert \Gamma _{1}\left( j\right) \right\vert \leq C_{u}G_{u}\left(
j\right) \left( 1+o_{j}\left( G_{u}\left( j\right) \right) \right) \leq
C_{u,M}B^{jq\left( \frac{u}{2}-1\right) }\left( 1+o_{j}\left( 1\right)
\right)
\end{equation*}%
and 
\begin{equation*}
\left\vert \Gamma _{1}\left( j\right) \right\vert \geq \left( 2\left\vert
G_{u}\left( j\right) \right\vert -\left\vert \sum_{r=1}^{u-1}\binom{u}{r}%
\left\vert G_{u-r}\left( j\right) G_{r}\left( j\right) \right\vert
\right\vert \right) \geq 2mc_{u}B^{jq\left( \frac{u}{2}-1\right) }\left(
1+o_{j}\left( 1\right) \right).
\end{equation*}%
Likewise, it holds that
\begin{equation*}
\left\vert \Gamma _{21}\left( j\right) \right\vert \leq G_{2u}\left(
j\right) \left( 1+o_{j}\left( G_{u}\left( j\right) \right) \right) \leq
C_{2u,M}^{\left( 1\right) }B^{j\frac{q}{2}\left( u-1\right) }\left(
1+o_{j}\left( 1\right) \right)
\end{equation*}
and
\begin{equation*}
\left\vert \Gamma _{21}\left( j\right) \right\vert \geq c_{2u,m}^{\left(
1\right) }B^{j\frac{q}{2}\left( u-1\right) }\left( 1+o_{j}\left( 1\right)
\right).
\end{equation*}%
Finally,
\begin{equation*}
c_{2u,m}^{\left( 2\right) }B^{j\frac{q}{2}\left( u-1\right) }\left(
1+o_{j}\left( 1\right) \right) \leq \left\vert \Gamma _{22}\left( j\right)
\right\vert \leq C_{2u,M}^{\left( 2\right) }B^{j\frac{q}{2}\left( u-1\right)
}\left( 1+o_{j}\left( 1\right) \right),
\end{equation*}%
as claimed.
\end{proof}
\begin{proposition}
\label{prop_cov}Let $\Sigma _{j,t}=\left\{ \Sigma _{j,t}\left(
k_{1},k_{2}\right) \colon k_{1},k_{2}=1,\ldots,d\right\} $ be a $d\times d$ positive
definite matrix such that 
\begin{equation*}
\Sigma _{j,t}\left( k_{1},k_{2}\right) = \mathbb{E}\left[ I_{1}\left(
g_{jk_{1},t}\left( z\right) \right) ,I_{1}\left( g_{jk_{2}}\left( z\right)
\right) \right]   =\left\langle g_{jk_{1}}\left( z\right) ,g_{jk_{2}}\left( z\right)
\right\rangle _{L^{2}\left( \mu _{t}\right) }. 
\end{equation*}
Then, there exists a constant $C_{\sigma ,M,\tau }$ such that 
\begin{equation*}
\Sigma _{j\left( t\right) ,t}\left( k_{1},k_{2}\right) -\delta
_{k_{1}}^{k_{2}}\leq \frac{C_{\sigma ,M,\tau }}{\left( 1+B^{\frac{q}{2}%
j}d\left( \xi _{jk_{1}},\xi _{jk_{2}}\right) \right) ^{\tau }}.
\end{equation*}%
Therefore, as $j\left( t\right) \underset{t\rightarrow \infty }{\rightarrow }%
\infty $, 
\begin{equation*}
\lim_{t\rightarrow \infty }\Sigma _{j\left( t\right) ,t}\left(
k_{1},k_{2}\right) =\delta _{k_{1}}^{k_{2}}.
\end{equation*}
\end{proposition}
\begin{proof}
Following \cite{PecZheng}, we have that for $1\leq k_{1},k_{2}\leq d$, 
\begin{eqnarray*}
\left\langle \widetilde{g}_{jk_{1}},\widetilde{g}_{jk_{2}}\right\rangle
_{L^{2}\left( \mu _{t}\right) } &=&\frac{1}{R_{t}^{3}\sigma ^{2}B^{jq\left( u-1\right) }}\int_{\mathbb{S}%
^{q}}\left( \int_{\mathbb{S}^{q}}h_{jk_{1}}\left( z_{1},z_{2}\right) \mu
_{t}\left( dz_{2}\right) \int_{\mathbb{S}^{q}}h_{jk_{2}}\left(
z_{1},z_{3}\right) \mu _{t}\left( dz_{3}\right) \right) \mu _{t}\left(
dz_{1}\right) \\
&=&\frac{1}{\sigma ^{2}B^{jq\left( u-1\right) }}\left[ \int_{\mathbb{S}%
^{q}}\left( \sum_{i_{1}=0}^{u}\binom{u}{i_{1}}\psi _{jk_{1}}^{u-i_{1}}\left(
z\right) G_{i_{1}}\left( j\right) \right) \left( \sum_{i_{2}=0}^{u}\binom{u}{%
i_{2}}\psi _{jk_{2}}^{u-i_{2}}\left( z\right) G_{i_{2}}\left( j\right)
\right) \mu \left( dz\right) \right] \\
& \leq & \frac{M}{\sigma ^{2}B^{jq\left( u-1\right) }}\sum_{i_{1}=0}^{u}%
\sum_{i_{2}=0}^{u}\binom{u}{i_{1}}\binom{u}{i_{2}}G_{i_{1}}\left( j\right)
G_{i_{2}}\left( j\right) \int_{\mathbb{S}^{q}}\psi _{jk_{1}}^{u-i_{1}}\left(
z\right) \psi _{jk_{2}}^{u-i_{2}}\left( z\right) dz \\
&\leq & C_{\tau ,M,u,\sigma }\frac{B^{\left( \frac{2u}{2}-1\right) qj}}{%
B^{jq\left( u-1\right) }}\left( \frac{1}{\left( 1+B^{\frac{q}{2}j}d\left(
\xi _{jk_{1}},\xi _{jk_{2}}\right) \right) ^{u\tau }}\right) \leq C_{\tau
,M,u,\sigma }\left( \frac{1}{\left( 1+B^{\frac{q}{2}j}d\left( \xi
_{jk_{1}},\xi _{jk_{2}}\right) \right) ^{u\tau }}\right).
\end{eqnarray*}
From Lemma \ref{lemma_int}, we hence have 
\begin{equation*}
\left\vert \left\langle \widetilde{g}_{ju},\widetilde{g}_{jv}\right\rangle
_{L^{2}\left( \mu _{t}\right) }\right\vert \leq \frac{C_{\tau ,M,u,\sigma }}{%
\left( 1+B^{\frac{d}{2}j}d\left( \xi _{jk_{1}},\xi _{jk_{2}}\right) \right)
^{u\tau }},
\end{equation*}%
as claimed.
\end{proof}
\subsection{Auxiliary results related to the proof of Theorem \ref{maintheorem2}}
\begin{lemma}
\label{corverysimilar}Let $\gamma _{j,q}$ be given by \eqref{gammabumbumconq}. We have that 
\begin{equation*}
\left\Vert \sum_{k}\psi _{jk}\otimes \psi _{jk}\right\Vert _{L^{2}\left( \mu
_{t}^{\otimes 2}\right) }^{2}=\gamma _{j,q}R_{t}^{2}B^{qj}.
\end{equation*}
\end{lemma}
\begin{proof}
In view of \eqref{uniformitynonzero}, we have%
\begin{equation*}
\left\Vert \sum_{k}\psi _{jk}\otimes \psi _{jk}\right\Vert _{L^{2}\left( \mu
_{t}^{\otimes 2}\right) }^{2} =\sum_{k_{1},k_{2}}\left( \int_{\mathbb{S}%
^{q}}\psi _{jk_{1}}\left( z\right) \psi _{jk_{2}}\left( z\right) \mu
_{t}\left( dz\right) \right) ^{2} =\left( \omega _{j}^{-1}R_{t}\right) ^{2}\sum_{k_{1},k_{2}}\left( \int_{%
\mathbb{S}^{q}}\psi _{jk_{1}}\left( z\right) \psi _{jk_{2}}\left( z\right)
dz\right) ^{2}.
\end{equation*}%
Hence, from Lemma \ref{rewritekernel} we obtain 
\begin{eqnarray*}
\sum_{k_{1},k_{2}}\left( \int_{\mathbb{S}^{q}}\psi _{jk_{1}}\left(
z\right) \psi _{jk_{2}}\left( z\right) dz\right) ^{2} &=& \sum_{k_{1},k_{2}}\int_{\mathbb{S}^{q}}\psi _{jk_{1}}\left( z_{1}\right)
\psi _{jk_{1}}\left( z_{2}\right) \psi _{jk_{2}}\left( z_{1}\right) \psi
_{jk_{2}}\left( z_{2}\right) dz_{1}dz_{2} \\
&=&\int_{\left( \mathbb{S}^{q}\right) ^{2}}\left( \sum_{\ell }b^{2}\left( 
\frac{\ell }{B^{j}}\right) \frac{\ell _{1}+\eta _{q}}{\eta _{q}\omega _{q}}%
\mathcal{C}_{\ell }^{\left( \eta _{q}\right) }\left( \left\langle
z_{1},z_{2}\right\rangle \right) \right) ^{2}dz_{1}dz_{2} \\
&=&\sum_{\ell _{1},\ell _{2}}\prod_{i=1,2}\left( b^{2}\left( \frac{\ell _{i}%
}{B^{j}}\right) \frac{\ell _{i}+\eta _{q}}{\eta _{q}\omega _{q}}\right)
\int_{\mathbb{S}^{q}\times \mathbb{S}^{q}}\mathcal{C}_{\ell _{1}}^{\left(
\eta _{q}\right) }\left( \left\langle z_{1},z_{2}\right\rangle \right) 
\mathcal{C}_{\ell _{2}}^{\left( \eta _{q}\right) }\left( \left\langle
z_{1},z_{2}\right\rangle \right) dz_{1}dz_{2} \\
&=&\sum_{\ell }b^{4}\left( \frac{\ell }{B^{j}}\right) \frac{\ell +\eta _{q}}{%
\eta _{q}\omega _{q}}\mathcal{C}_{\ell }^{\left( \eta _{q}\right) }\left(
1\right) \int_{\mathbb{S}^{q}}dz.
\end{eqnarray*}%
Using \eqref{gegenbauerin1}, we obtain 
\begin{equation*}
\left\Vert \sum_{k}\psi _{jk}\otimes \psi _{jk}\right\Vert _{L^{2}\left( \mu
_{t}^{\otimes 2}\right) }^{2} = R_{t}^{2}\omega _{q}^{-1}\sum_{\ell
}b^{4}\left( \frac{\ell }{B^{j}}\right) \frac{\ell +\eta _{q}}{\eta
_{q}\omega _{q}}\binom{\ell +2\eta _{q}-1}{\ell } = R_{t}^{2}B^{qj}\gamma _{j,q}.
\end{equation*}
\end{proof}
\begin{lemma}
\label{lemmaverysimilar} Let $\gamma _{j,q}$ be given by \eqref{gammabumbumconq}. Then, there exist positive constants $c_{1},c_{2}>0$
such that, for all $j>0$%
\begin{equation}
c_{1}\leq \gamma _{j,q}\leq c_{2}.  \label{gammasmallbound}
\end{equation}%
Moreover, as $j\rightarrow \infty $, $\gamma _{j,q}\rightarrow \gamma _{q}$, where $\gamma _{q}$ is given by \eqref{gammabumbum}.
\end{lemma}

\begin{proof}
The inequality \eqref{gammasmallbound} is easily proved by rewriting $\gamma
_{j,q}$ in the framework of the Remark \ref{stilltobewritten} and using
Lemma \ref{lemmawithsinteger}. Indeed, for Lemma \ref{rewritekernel} and
considering \eqref{gegenbauerin1} 
\begin{equation*}
\gamma _{j,q} = \left( \omega _{q}B^{qj}\right) ^{-1}\sum_{\ell
}b^{4}\left( \frac{\ell }{B^{j}}\right) \frac{\ell +\eta _{q}}{\eta
_{q}\omega _{q}}\mathcal{C}_{\ell }^{\left( \eta _{q}\right) }\left(
1\right) =\left( \eta _{q}\omega _{q}^{2}\right) ^{-1}\frac{1}{B^{j}}\sum_{\ell
}b^{4}\left( \frac{\ell }{B^{j}}\right) \frac{\ell +\eta _{q}}{B^{j}}\frac{%
\mathcal{C}_{\ell }^{\left( \eta _{q}\right) }\left( 1\right) }{B^{j\left(
q-2\right) }}.
\end{equation*}
Now, for all the values of $j$, we have that 
\begin{equation}
c_{1}^{\prime }=B^{-\left( q-2\right) }\leq \frac{\mathcal{C}_{\ell
}^{\left( \eta _{q}\right) }\left( 1\right) }{B^{j\left( q-2\right) }}=\frac{%
1}{B^{j\left( q-2\right) }}\left( 
\begin{array}{c}
\ell +q-2 \\ 
\ell 
\end{array}%
\right) \leq B^{q-2}=c_{2}^{\prime }.  \label{mart1}
\end{equation}%
Likewise,
\begin{equation}
B^{-1}+O\left( B^{-j}\right) \leq \frac{\ell +\eta _{q}}{B^{j}}\leq
B+O\left( B^{-j}\right),  \label{mart2}
\end{equation}%
and 
\begin{equation}
c_{1}^{\prime \prime }\leq \frac{1}{B^{j}}\sum_{\ell }b^{4}\left( \frac{\ell 
}{B^{j}}\right) \leq c_{2}^{\prime \prime }.  \label{mart3}
\end{equation}%
Combining \eqref{mart1}, \eqref{mart2} and \eqref{mart3}, we obtain \eqref{gammasmallbound}. On the other hand, up to factors of smaller order, \eqref{gammabumbumconq} is a Riemann sum of the integral in \eqref{gammabumbum}, so that
\begin{equation*}
\lim_{j\rightarrow \infty }\left( \omega _{q}B^{qj}\right) ^{-1}\sum_{\ell
}b^{4}\left( \frac{\ell }{B^{j}}\right) \frac{\ell +\eta _{q}}{\eta
_{q}\omega _{q}}\binom{\ell +q-2}{\ell} =\gamma _{q}.
\end{equation*}
\end{proof}

\end{document}